\RequirePackage{amsmath}
\documentclass[a4paper, 10pt]{amsart}
\usepackage{amssymb, url, color, pb-diagram, graphicx, amscd, pb-diagram,mathrsfs}
\usepackage[colorlinks=true, bookmarks=true, pdfstartview=FitH, pagebackref=true]{hyperref}
\usepackage[nodayofweek]{datetime}
\usepackage{enumerate,stmaryrd} 
\usepackage{euscript}
\usepackage{verbatim}
\usepackage{graphicx}
\usepackage{times}
\usepackage{arydshln} % for using \hdashline only
\usepackage{setspace}
\usepackage[space, compress, sort]{cite}
\usepackage{empheq}
\usepackage[usenames,dvipsnames]{xcolor}

\numberwithin{equation}{section}

\theoremstyle{plain}
\newtheorem{thm}{Theorem}[section]
\newtheorem{theorem}[thm]{Theorem}

\newtheorem{lemma}[thm]{Lemma}

\newtheorem{proposition}[thm]{Proposition}

\newcommand{\definedas}{\coloneqq}

\newcommand{\bR}{\mathbb{R}}
\newcommand{\bN}{\mathbb{N}}
\newcommand{\bL}{\mathbb{L}}

\newcommand{\bE}{\mathbf{E}}

\newcommand{\cY}{\mathcal{Y}}

\newcommand{\ghat}{{\widehat{g}}}
\newcommand{\Khat}{\widehat{K}}

\newcommand{\phitil}{\widetilde{\phi}}
\newcommand{\sigmatil}{\widetilde{\sigma}}
\newcommand{\Wtil}{\widetilde{W}}

\newcommand{\Cbar}{\overline{C}}
\newcommand{\Wbar}{\overline{W}}
\newcommand{\gbar}{\overline{g}}
\newcommand{\qbar}{\overline{q}}

\newcommand{\Wring}{\mathring{W}}
\newcommand{\Lring}{\mathring{L}}

\DeclareMathOperator{\tr}{tr}
\DeclareMathOperator{\divg}{div}
\DeclareMathOperator{\vol}{Vol}

\let\<\langle 
\let\>\rangle

\newcommand{\scal}{\mathrm{Scal}}
\newcommand{\ric}{\mathrm{Ric}}
\newcommand{\lie}{\mathcal{L}}
\newcommand{\scaltil}{\widetilde{\scal}}

\def\into{\hookrightarrow}

\DeclareMathOperator{\DeltaL}{\Delta_{\bL, \eta}}

\begin{document}
%%%%%%%%%%%%%%%%%%%%%%%%%%%%%%%%%%%%%%%%%%%%%%%%%%%%%%%%%%%%%%%%%

%%%%%%%%%%%%%%%%%%%%%%%%%%%%%%%%%%%%%%%%%%%%%%%%%%%%%%%%%%%%%%%%%
\title[Solutions to constraint equations with vanishing Yamabe invariant]
{Solutions to the Einstein constraint equations with a small
TT-tensor and vanishing Yamabe invariant}

\begin{abstract}
In this note we prove an existence result for the Einstein conformal
constraint equations for metrics with vanishing Yamabe invariant
assuming that the TT-tensor is small in $L^2$.
\end{abstract}

\author[R. Gicquaud]{Romain Gicquaud}
\address[R. Gicquaud]{
  Institut Denis Poisson\\
  Universit\'e de Tours\\
  Parc de Grandmont\\ 37200 Tours \\ FRANCE}
\email{romain.gicquaud@lmpt.univ-tours.fr}



\date{14th February, 2018}
\maketitle
\tableofcontents
%%%%%%%%%%%%%%%%%%%%%%%%%%%%%%%%%%%%%%%%%%%%%%%%%%%%%%%%%%%%%%%%%%%%%%%%%

%%%%%%%%%%%%%%%%%%%%%%%%%%%%%%%%%%%%%%%%%%%%%%%%%%%%%%%%%%%%%%%%%%%%%%%%%
\section{Introduction}
%%%%%%%%%%%%%%%%%%%%%%%%%%%%%%%%%%%%%%%%%%%%%%%%%%%%%%%%%%%%%%%%%%%%%%%%%
The conformal method and one of its generalization, the conformal thin
sandwich (CTS) method (described e.g. in \cite{BartnikIsenberg} or
\cite{MaxwellConformalMethod}) are historically the main methods to solve
the Einstein constraint equations, despite recent evidences that they fail
at parameterizing correctly the full set of initial data (see e.g.
\cite{MaxwellConformalParameterization, HolstMeierNonUniqueness,
DiltsHolstKozarevaMaxwell}).

Initial data for the Cauchy problem are generally given as a triple
$(M, \ghat, \Khat)$, where $M$ is a n-dimensional manifold, $\ghat$ is a
metric on $M$ and $\Khat$ is a symmetric 2-tensor that correspond respectively
to the metric induced by the spacetime (we are to find) metric on $M$ and
the second fundamental form of $M$ as a hypersurface in the spacetime. The
interested reader can consult e.g. \cite{RingstromCauchyProblem} for more
information.

The strategy of the conformal method and of the CTS method is to decompose
in a certain manner $(M, \ghat, \Khat)$ as a given part and an unknown part
that has to be adjusted in order to fulfill the constraint equations. To keep
things simple, we will consider only the vacuum case and restrict to compact
Cauchy surfaces $M$. We fix
\begin{itemize}
 \item a compact Riemannian manifold $(M, g)$ of dimension $n \geq 3$,
 \item a function $\tau: M \to \bR$,
 \item a symmetric traceless 2-tensor $\sigma$ such that $\divg_g \sigma = 0$
(such a tensor will be called a TT-tensor in what follows),
 \item a positive function $\eta: M \to \bR_+$,
\end{itemize}
and seek for
\begin{itemize}
 \item a positive function $\phi: M \to \bR_+$,
 \item a vector field $W$,
\end{itemize}
so that
\begin{equation}\label{eqDecomposition}
 \ghat \definedas \phi^{N-2} g\quad\text{and}\quad\Khat \definedas \frac{\tau}{n} \ghat + \phi^{-2} \left(\sigma + \frac{1}{2\eta}\bL W\right)
\end{equation}
satisfy the constraint equations
\begin{subequations}\label{eqConstraintEquations}
\begin{empheq}[left=\empheqlbrace]{align}
\label{eqHamiltonian}
\scal_\ghat + (\tr_\ghat \Khat )^2 - |\Khat|_\ghat^2 & = 0,\\
\label{eqMomentum}
\divg_\ghat \Khat - d (\tr_\ghat \Khat ) & = 0.
\end{empheq}
\end{subequations}
Here we have introduced the following notations:
\[
 N \definedas \frac{2n}{n-2}\quad\text{and}\quad \bL W \definedas \lie_W g - \frac{\tr_g \lie_W g}{n} g,
\]
where $\lie$ denotes the Lie derivative.
The operator $\bL$ is commonly known as the conformal Killing operator
or as the Alhfors operator. Note that $\tau = \ghat^{ij} \Khat_{ij}$ so
$\tau$ corresponds to the mean curvature of the embedding of $M$ into
the spacetime.

The decomposition \eqref{eqDecomposition} relies on York's splitting
of symmetric 2-tensors \cite{YorkDecomposition}. TT-tensors were introduced
first by R. Arnowitt, S. Deser and C. Misner in 1962 (see the reprint of
this article in \cite{ArnowittDeserMisner}). We refer the reader to
\cite{MaxwellConformalMethod} for more information about the history of the
conformal method and of the conformal thin sandwich.
% Note also that there is a slight difference
% between our convention of choosing $\sigma$ to be TT and the one
% in \cite{BartnikIsenberg, MaxwellConformalMethod}. Since we shall assume
% that $g$ has no non-zero conformal Killing vector fields (i.e. vector fields
% $V$ such that $\bL V = 0$), the distinction is irrelevant (see Proposition \ref{propVector}).

The system \eqref{eqConstraintEquations} is equivalent to the following:
\begin{subequations}\label{system}
\begin{empheq}[left=\empheqlbrace]{align}
-\frac{4(n-1)}{n-2} \Delta_g \phi + \scal \phi &= - \frac{n-1}{n} \tau^2 \phi^{N-1} + \left|\sigma + \frac{1}{2\eta}\bL_g W \right|_g^2 \phi^{-N-1}, \label{eqLichnerowicz}\\
\DeltaL W &= \frac{n-1}{n} \phi^N \nabla\tau, \label{eqVector}
\end{empheq}
\end{subequations}
where we set
\[
 \DeltaL \definedas - \frac{1}{2} \bL^* \left(\frac{1}{2\eta} \bL \cdot\right).
\]
Equation \eqref{eqLichnerowicz} is commonly known as the Lichnerowicz
equation while Equation \eqref{eqVector} bears no particular name, we
will call it the vector equation. Hence, solving \eqref{system} is
equivalent to solving \eqref{eqConstraintEquations}.

The conformal method corresponds to the particular choice $2\eta \equiv 1$
in the previous equations. As indicated in \cite{BartnikIsenberg},
allowing for more general $\eta$ in \eqref{system} does not introduce
new technical difficulties so theoretical studies have mostly concentrated
on the conformal method.

Initial work was limited to the constant mean
curvature (CMC) case (i.e. constant $\tau$) and to the near-CMC case.
But two constructions were introduced in 2009 by M. Holst, G. Nagy, G. Tsogtgerel
and D. Maxwell (HNTM), see \cite{HNT1, HNT2, MaxwellNonCMC}, and in 2011 by M. Dahl,
E. Humbert and the author in \cite{DahlGicquaudHumbert} to solve \eqref{system}.
The interested reader can consult \cite{GicquaudNgo, Nguyen2} for an overview
and a comparison of both techniques.

We will focus on the HNTM method. It requires two things: that the Yamabe
invariant $\cY_g$ of $g$ is positive (see e.g. \cite{LeeParker} for the
definition of the Yamabe invariant and the solution of the related Yamabe
problem) and that $\sigma$ is non zero but
small in a certain sense (see also \cite{Nguyen}). This construction was
interpreted as perturbative in a non-trivial sense in \cite{GicquaudNgo}.
Despite the fact that the point of view introduced in \cite{GicquaudNgo} gives
a result weaker than the original one in \cite{MaxwellNonCMC}, it provides
a quick way to test whether the HNTM method works in more general
situations. This has been used in \cite{GicquaudNguyen} for the Einstein-scalar
field conformal method and in \cite{HolstMaxwellMazzeo} for variants of the
conformal method introduced by D. Maxwell in \cite{MaxwellInitialData}.

In this paper, we show that the HNTM construction extends to
the case $\cY_g = 0$ at the price of imposing a soft (explicit) condition
on $\tau$, see \eqref{eqCondition}. The main difficulty here is that the
conformal Laplacian
\[
 - \frac{4(n-1)}{n-2} \Delta_g + \scal
\]
has a 1-dimensional kernel so the behavior of $\phi$ in the direction of
this kernel is different than in the ($L^2$-)orthogonal direction.

The outline of the paper is as follows. Section \ref{secPrelim} contains
existence and uniqueness results for the Lichnerowicz equation and for the
vector equation in a weak regularity context.
Section \ref{secImplicit} follows the construction in \cite{GicquaudNgo}.
This gives an idea of what goes on and prepares for Section \ref{secSmallTT}
where we prove existence of solutions to \eqref{system} when $\sigma$ is small
in the spirit of \cite{MaxwellNonCMC, Nguyen, GicquaudNguyen}. 

The main difference between Theorem \ref{thmImplicit} and Theorem
\ref{thmSmallTT} is that, in \ref{thmImplicit}, we have no control on how
$\lambda_0$ depends on $\sigmatil$ so the theorem gives a weaker result existence,
yet the proof is based on the implicit function theorem so is constructive.
The proof of Theorem \ref{thmSmallTT} however is based on the Schauder
fixed point theorem and is non-constructive by essence.\\

\noindent\textbf{Acknowledgments:} 
The author is grateful to Cang Nguyen for useful discussion and careful
proofreading of the paper. He also thanks Stanley Deser for pointing the
reference \cite{ArnowittDeserMisner}. Last but not least, the author's
thoughts are with Jim Isenberg and his relatives after this tragic accident.

%%%%%%%%%%%%%%%%%%%%%%%%%%%%%%%%%%%%%%%%%%%%%%%%%%%%%%%%%%%%%%%%%%%%%%%%%
\section{Preliminaries}\label{secPrelim}
%%%%%%%%%%%%%%%%%%%%%%%%%%%%%%%%%%%%%%%%%%%%%%%%%%%%%%%%%%%%%%%%%%%%%%%%%
The aim of this section is to reprove well known existence results in a
weak regularity context. Here and in what follows, we fix a value $p > n$.

We let $\phi_0$ denote the unique positive function such that
$\|\phi_0\|_{L^2(M, \bR)} = 1$ and
\[
 -\frac{4(n-1)}{n-2} \Delta \phi_0 + \scal \phi_0 = 0.
\]

\begin{proposition}\label{propLichnerowicz}
 Given $g \in W^{2, p/2}(M, S_2M)$, $\tau \in L^p(M, S_2M)$ and $A \in L^p(M, \bR)$ both non zero,
there exists a unique positive function $\phi \in W^{2, p/2}(M, \bR)$
solving the Lichnerowicz equation 
\begin{equation}\label{eqLichnerowicz2}
-\frac{4(n-1)}{n-2} \Delta \phi + \scal \phi
= -\frac{n-1}{n} \tau^2 \phi^{N-1} + A^2 \phi^{-N-1}.
\end{equation}
Further the mapping $A \mapsto \phi$ is continuous from $L^p(M, \bR)$ to $W^{2, p/2}(M, \bR)$.
\end{proposition}

It should be noted that, if either $A \equiv 0$ or $\tau \equiv 0$, there
cannot be any non-zero solution for a simple reason. Multiplying the Lichnerowicz
equation by $\phi_0$ and integrating over $M$, the conformal Laplacian disappears
by (formal) self-adjointness leaving the following equality
\[
 \frac{n-1}{n} \int_M \tau^2 \phi_0 \phi^{N-1} d\mu^g = \int_M A^2 \phi_0 \phi^{-N-1} d\mu^g.
\]
If $A$ or $\tau$ vanishes, this identity leads to a contradiction if
Equation \eqref{eqLichnerowicz2} admits a positive solution $\phi$.
There is one exception to this fact, namely when both $\tau$ and $A$
vanish (compare with \cite{Isenberg}).
In this case the solutions to \eqref{eqLichnerowicz2} are 
the $\lambda \phi_0$, $\lambda \in \bR_+$. We will not consider these
cases anymore.

\begin{proof}[Proof of Proposition \ref{propLichnerowicz}]
We use a variational approach. Note that the functional
\begin{equation}\label{eqDefFunctional}
 I(\phi) \definedas \int_M \left(\frac{2(n-1)}{n-2} |d\phi|^2 + \frac{\scal}{2} \phi^2 + \frac{n-1}{Nn} \tau^2 \phi^N + \frac{A^2}{N \phi^N}\right) d\mu^g
\end{equation}
is ill-defined on $W^{1, 2}(M, \bR)$ since the term $\tau^2 \phi^N$ does not belong to $L^1(M, \bR)$.
For any positive integer $k$, we set
$\tau_k \definedas \min\{\tau, k\} \in L^\infty$ so
that $\tau_k \to \tau$ in $L^p$ and $\epsilon_k \definedas 1/k$, we introduce
the family of functionals
\[
 I_k(\phi) \definedas \int_M \left(\frac{2(n-1)}{n-2} |d\phi|^2 + \frac{\scal+\epsilon_k}{2} \phi^2 + \frac{n-1}{Nn} \tau_k^2 \phi^N + \frac{A^2}{N(\phi+\epsilon_k)^N}\right) d\mu^g.
\]
$I_k$ is well defined, continuous and convex on the closed set
\[
 C_k \definedas \{\phi \in W^{1, 2}(M, \bR), \phi \geq \epsilon_k/2 \text{ a.e.}\}
\]
(details can be found in \cite{GicquaudNguyen}). We claim that there
exists $\mu_k > 0$ so that $I_k(\phi) \geq \mu_k \|\phi\|_{W^{1, 2}(M, \bR)}^2$
for all $\phi \in C_k$. Indeed, it suffices to prove that there exists
$\mu_k > 0$ such that
\[
 \mu_k \|\phi\|^2_{W^{1, 2}(M, \bR)} \leq \int_M \left(\frac{2(n-1)}{n-2} |d\phi|^2 + \frac{\scal+\epsilon_k}{2} \phi^2\right) d\mu^g
\]
for all $\phi \in W^{1, 2}(M, \bR)$. Since $\scal \in L^{p/2}(M, \bR)$ with $p > n$,
we have
\begin{align*}
 &\left\|\left(\frac{\scal+\epsilon_k}{2} - \frac{2(n-1)}{n-2} \right)\phi^2\right\|_{L^1(M, \bR)}\\
 &\qquad \leq \frac{1}{2} \left[\left\|\scal\right\|_{L^{p/2}(M, \bR)} + \left(\frac{2(n-1)}{n-2}- \epsilon_k\right) \vol(M, g)^{2/p}\right] \|\phi\|^2_{L^q(M, \bR)},
\end{align*}
with $q = \frac{2p}{p-2} < N$. Set
\[
 c \definedas \frac{1}{2} \left[\left\|\scal\right\|_{L^{p/2}(M, \bR)} + \frac{2(n-1)}{n-2} \vol(M, g)^{2/p}\right]
\]
so that 
\[
 \left\|\left(\frac{\scal+\epsilon_k}{2} - \frac{2(n-1)}{n-2} \right)\phi^2\right\|_{L^1(M, \bR)} \leq c \|\phi\|^2_{L^q(M, \bR)}.
\]
By interpolation, for any $\epsilon > 0$, there is a constant $\Lambda_\epsilon > 0$
such that, for all $\phi \in W^{1, 2}(M, \bR)$,
\[
 \|\phi\|^2_{L^q(M, \bR)} \leq \epsilon \|\phi\|^2_{W^{1, 2}(M, \bR)} + \Lambda_\epsilon \|\phi\|_{L^2(M, \bR)}^2.
\]
Choose $\epsilon = (n-1)/(c(n-2))$. We have
\begin{align*}
&\int_M \left(\frac{2(n-1)}{n-2} |d\phi|^2 + \frac{\scal+\epsilon_k}{2} \phi^2\right) d\mu^g\\
&\qquad = \frac{2(n-1)}{n-2} \int_M \left(|d\phi|^2 + \phi^2 \right)d\mu^g + \int_M\left(\frac{\scal+\epsilon_k}{2} - \frac{2(n-1)}{n-2}\right) \phi^2 d\mu^g\\
&\qquad \geq \frac{2(n-1)}{n-2} \|\phi\|_{W^{1, 2}(M, \bR)}^2 - \left\|\left(\frac{\scal+\epsilon_k}{2} - \frac{2(n-1)}{n-2}\right) \phi^2\right\|_{L^1(M, \bR)}\\
&\qquad \geq \frac{2(n-1)}{n-2} \|\phi\|_{W^{1, 2}(M, \bR)}^2 - c \left\|\phi\right\|_{L^q(M, \bR)}^2\\
&\qquad \geq \frac{n-1}{n-2} \|\phi\|_{W^{1, 2}(M, \bR)}^2 - c \Lambda_\epsilon \left\|\phi\right\|_{L^2(M, \bR)}^2.
\end{align*}
However, from the definition of the Yamabe invariant, we have
\[
 \int_M \left(\frac{2(n-1)}{n-2} |d\phi|^2 + \frac{\scal}{2} \phi^2\right) d\mu^g \geq 0
\]
for all functions $\phi \in W^{1, 2}(M, \bR)$. As a consequence,
\[
 \int_M \left(\frac{2(n-1)}{n-2} |d\phi|^2 + \frac{\scal+\epsilon_k}{2} \phi^2\right) d\mu^g \geq \frac{\epsilon_k}{2} \|\phi\|_{L^2(M, \bR)}^2.
\]
Finally, combining both estimates, we obtain
\begin{align*}
&\left(1 + \frac{2c \Lambda_\epsilon}{\epsilon_k}\right) \int_M \left(\frac{2(n-1)}{n-2} |d\phi|^2 + \frac{\scal+\epsilon_k}{2} \phi^2\right) d\mu^g\\
&\qquad \geq \frac{n-1}{n-2} \|\phi\|_{W^{1, 2}(M, \bR)}^2 - c \Lambda_\epsilon \left\|\phi\right\|_{L^2(M, \bR)}^2 + \frac{2c \Lambda_\epsilon}{\epsilon_k} \frac{\epsilon_k}{2} \|\phi\|_{L^2(M, \bR)}^2\\
&\qquad \geq \frac{n-1}{n-2} \|\phi\|_{W^{1, 2}(M, \bR)}^2.
\end{align*}
This proves that for all $\phi \in C_k$, we have $I_k(\phi) \geq \mu_k \|\phi\|_{W^{1, 2}(M, \bR)}^2$
with
\[
 \mu_k = \frac{n-1}{n-2} \left(1 + \frac{2c \Lambda_\epsilon}{\epsilon_k}\right)^{-1}.
\]
In particular, any minimizing sequence $(\phi_i)$ for $I_k$ is bounded in $W^{1, 2}(M, \bR)$
since the norm of $\phi_i$ eventually becomes less that $\mu_k^{-1} I_k(1)$.
It is then a standard fact that there exists a minimizer $\phi_k$ for $I_k$ in $C_k$ and, since
$I_k$ is strictly convex, $\phi_k$ is unique.

At this point, we remark that for any $\phi \in C_k$, $I_k(|\phi|) \leq I_k(\phi)$,
so $\phi_k \geq 0$.
It should be noted that $C_k$ has empty interior in $W^{1, 2}(M, \bR)$ so it makes no
sense to speak of the (G\^ateau) differential of $I_k$. However, if $f$ is
a smooth (more generally, if $f \in W^{1, 2}(M, \bR) \cap L^\infty(M, \bR)$), we can define
the directional derivative of $I_k$ in the direction $f$. This is sufficient
to conclude that $\phi_k$ is a weak solution to
\begin{equation}\label{eqLichnerowiczK}
 -\frac{4(n-1)}{n-2} \Delta \phi + (\scal+\epsilon_k) \phi
= -\frac{n-1}{n} \tau_k^2 \phi^{N-1} + A^2 (\phi+\epsilon_k)^{-N-1}.
\end{equation}
Note that the right hand side of this equation belongs to $L^{p/2}(M, \bR)$ so,
by elliptic regularity, we have $\phi_k \in W^{2, p/2}(M, \bR)$ and from Harnack's
inequality (see e.g. \cite{TrudingerMeasurable}), $\phi_k > 0$.

We now let $k$ tend to infinity. We first prove that the functions $\phi_k$
are uniformly bounded from below by constructing suitable subsolutions.
Let $u \in W^{2, p/2}(M, \bR)$ denote the solution to the following equation
\[
 - \frac{4(n-1)}{n-2} \Delta u + \scal u + \frac{n-1}{n} \tau^2 u = A^2.
\]
$u$ can be obtained by minimizing the functional
\[
 J(u) \definedas \int_M \left[\frac{2(n-1)}{n-2} |du|^2 + \left(\frac{\scal}{2} + \frac{n-1}{2n} \tau_0^2\right) u^2 - A^2 u\right] d\mu^g,
\]
and, as before $J(|u|) \leq J(u)$ so $u \geq 0$ and $u > 0$ by Harnack's
inequality (Note that we overcame the ill-definiteness of the $\tau$-term by changing
the exponent). Let $u_k$ denote the solution to
\[
 - \frac{4(n-1)}{n-2} \Delta u_k + (\scal + \epsilon_k) u_k + \frac{n-1}{n} \tau^2 u_k = A^2.
\]
We let the reader convince himself that $u_k \in W^{2, p/2}(M, \bR)$, $u_k > 0$ and
$u_k \to u$ in $W^{2, p/2}(M, \bR)$ as $k$ tends to infinity. Since $W^{2, p/2}(M, \bR)$
embeds continuously in $L^\infty(M, \bR)$, there exist constants
$c_-, c_+ > 0$ such that
\[
 c_-\leq u_k \leq c_+
\]
for all $k$.
We now look for $\lambda_- > 0$ so that $\phi_{k,-} = \lambda_- u_k$
is a subsolution to Equation \eqref{eqLichnerowiczK}. We want
\begin{equation}\label{eqSubsolution}
 -\frac{4(n-1)}{n-2} \Delta \phi_{k,-} + (\scal+\epsilon_k) \phi_{k,-}
\leq -\frac{n-1}{n} \tau_k^2 \phi_{k,-}^{N-1} + A^2 (\phi_{k,-}+\epsilon_k)^{-N-1}.
\end{equation}
Equivalently, 
\[
 \lambda_- \left(-\frac{4(n-1)}{n-2} \Delta u_k + (\scal+\epsilon_k) u_k\right) + \frac{n-1}{n} \tau_k^2 \lambda_-^{N-1} u_k^{N-1}
\leq A^2 (\lambda_- u_k+\epsilon_k)^{-N-1}
\]
which can be rewritten as follows:
\[
 \frac{n-1}{n} \tau_k^2 \lambda_-^{N-1} u_k^{N-1} - \frac{n-1}{n} \tau^2 \lambda_- u_k
\leq A^2 (\lambda_- u_k+\epsilon_k)^{-N-1} - \lambda_- A^2.
\]
Since $\tau_k \leq \tau$, the left hand side is non-positive if
$\lambda_-^{N-1} u_k^{N-1} \leq \lambda_- u_k$, i.e. $\lambda_- \leq (c_+)^{-1}$.
On the other hand, the right hand side is non-negative if
\[
 \lambda_- (\lambda_- u_k + \epsilon_k)^{N+1} \leq 1
\]
Since $\epsilon_k \leq 1$, we see that the previous inequality holds
when $\lambda_- \leq (1 + c_+)^{-N-1}$. We have proven that if
\[
 \lambda_- \leq \min\{(c_+)^{-1}, (1 + c_+)^{-N-1} \},
\]
$\phi_{-, k}$ is a subsolution to Equation \eqref{eqLichnerowiczK}.
We now prove that $\phi_k \geq \phi_{-, k}$. We compute the difference
between \eqref{eqLichnerowiczK} and \eqref{eqSubsolution}, multiply
it by $(\phi_k - \phi_{-, k})_- = \min\{0, \phi_k - \phi_{-, k}\}$
and integrate over $M$:
\begin{align*}
&\int_M \left(\frac{4(n-1)}{n-2} \left|d (\phi_k - \phi_{-, k})_-\right|^2 + (\scal+\epsilon_k) \left|(\phi_k - \phi_{-, k})_-\right|^2\right) d\mu^g\\
&\qquad \leq - \frac{n-1}{n} \int_M \tau_k^2 \left(\phi_k^{N-1} - \phi_{-,k}^{N-1}\right) (\phi_k - \phi_{-, k})_- d\mu^g\\
&\qquad\qquad + \int_M A^2 \left[(\phi_k+\epsilon_k)^{-N-1} - (\phi_{-,k}+\epsilon_k)^{-N-1}\right] (\phi_k - \phi_{-, k})_- d\mu^g.
\end{align*}
The right hand side is non-positive while the left hand side is non-negative.
This imposes that
\begin{align*}
 &\mu_k \|(\phi_k - \phi_{-, k})_-\|_{W^{1, 2}(M, \bR)}\\
 &\qquad \leq \int_M \left(\frac{4(n-1)}{n-2} \left|d (\phi_k - \phi_{-, k})_-\right|^2 + (\scal+\epsilon_k) \left|(\phi_k - \phi_{-, k})_-\right|^2\right) d\mu^g = 0
\end{align*}
So $(\phi_k - \phi_{-, k})_- \equiv 0$, which means that
$\phi_k \geq \phi_{-, k} \geq \lambda_- c_- > 0$.
This ends the proof of the fact that the functions $\phi_k$
are uniformly bounded from below. Let $\phi_+ \in W^{2, p/2}(M, \bR)$
denote the positive solution to
\[
 - \frac{4(n-1)}{n-2} \Delta \phi_+ + \scal \phi_+ + \frac{n-1}{n} \tau_1^2 \phi_+^{N-1} = \frac{A^2}{(\lambda_- c_-)^{N+1}},
\]
(we remind the reader that $\tau_1 = \min\{\tau, 1\}$). By similar arguments,
we can prove that $\phi_k \leq \phi_+$ so the sequence of functions $\phi_k$
is uniformly bounded from above and from below:
\[
 \lambda_- c_- \leq \phi_k \leq \max \phi_+.
\]
We rewrite Equation \eqref{eqLichnerowiczK} as
\begin{equation}\label{eqEstimate2}
-\frac{4(n-1)}{n-2} \Delta \phi_k + \phi_k = (1 - \scal+\epsilon_k) \phi_k
-\frac{n-1}{n} \tau_k^2 \phi_k^{N-1} + A^2 (\phi_k+\epsilon_k)^{-N-1},
\end{equation}
and notice that the right hand side is bounded in $L^{p/2}(M, \bR)$, so, by elliptic
regularity, $(\phi_k)_k$ is uniformly bounded in $W^{2, p/2}(M, \bR)$.
Since $W^{2, p/2}(M, \bR)$ compactly embeds in $L^\infty(M, \bR)$, we can assume that
$\phi_k$ converges strongly to some $\phi_\infty$ in $L^\infty(M, \bR)$ and, from 
\eqref{eqEstimate2}, $\phi_\infty \in W^{2, p/2}(M, \bR)$ solves the Lichnerowicz
equation \eqref{eqLichnerowicz2}.

Now the functional $I$ introduced in \eqref{eqDefFunctional} makes perfect
sense on the (open) subset $\Omega_+ = \{\phi \in W^{2, p/2}(M, \bR), \phi > 0\}$.
It is differentiable and strictly convex. Furthermore $\phi_\infty$ is a
critical point for $I$ on $\Omega_+$. So it must be the unique minimum of $I$
on $\Omega_+$. Since a strictly convex functional can only have a single critical
point and critical points of $I$ are exactly the solutions of \eqref{eqLichnerowicz2},
we conclude that $\phi_\infty$ is the unique solution to \eqref{eqLichnerowicz2}.

Continuity of $\phi_\infty$ with respect to $A$ follows from the implicit function
theorem in a way that is similar to the one presented in the next section so we omit
the proof of it.
\end{proof}

\begin{proposition}\label{propVector}
 Let $g \in W^{2, p/2}(M, S_2M)$, $\eta \in W^{1, p}(M, \bR)$, $\eta > 0$ and
$\xi \in L^q(M, TM)$ be given, for some $q \in (1, p/2)$. Assume that $g$ has
no conformal Killing vector fields, i.e. no non-trivial vector field $V$ such
that $\bL V = 0$. There exists a unique $W \in W^{2, q}(M, TM)$ solving
\begin{equation}\label{eqVectLapl}
 \DeltaL W = \xi.
\end{equation}
Further, the mapping $\xi \mapsto W$ is continuous.
\end{proposition}

\begin{proof}
 As in the proof of the previous proposition, we introduce the functional
\[
 J(W) \definedas \frac{1}{2} \int_M \frac{\left|\bL W\right|^2}{2\eta} d\mu^g - \int_M \left\<W, \xi\right\> d\mu^g
\]
for any $W \in W^{1, 2}(M, TM)$.
Since, in any coordinate system
\[
 \nabla_i W_j = \partial_i W_j - \Gamma^k_{ij} W_k,
\]
where $\Gamma^k_{ij} = \frac{1}{2} g^{kl}\left(\partial_i g_{lj} + \partial_j g_{il} - \partial_l g_{ij}\right) \in W^{1, p/2} \subset L^n$
denotes the Christoffel symbol of $g$, we have that, for any $W \in W^{1, 2}(M, TM)$, $\Gamma^k_{ij} W_k \in L^2(M, \bR)$ as a sum of products of
functions in $L^n$ and in $L^N(M, \bR)$. Hence, $\nabla W \in L^2(M, T^{\otimes 2}M)$ and, in particular, $\bL W \in L^2(M, S_2M)$.

We claim that the quadratic part of $J$ is coercive on $W^{1, 2}(M, TM)$.
Indeed, there exists a constant $\Lambda > 0$ so that $2\eta \leq \Lambda$.
It follows from the Bochner formula for $\bL$ that
\begin{align*}
\frac{1}{2} \int_M \frac{\left|\bL W\right|^2}{2\eta} d\mu^g
  &\geq \frac{1}{2\Lambda} \int_M \left|\bL W\right|^2 d\mu^g\\
  &\geq \frac{1}{\Lambda} \int_M \left[ \left|\nabla W\right|^2 + \left(1 - \frac{2}{n}\right) \left(\divg W\right)^2 - \ric(W,W)\right] d\mu^g\\
  &\geq \frac{1}{\Lambda} \left[\int_M \left|\nabla W\right|^2d\mu^g - \left\|\ric\right\|_{L^{p/2}(M, S_2M)} \left\|W\right\|^2_{L^q(M, TM)}\right],
\end{align*}
where $q = 2p/(p-2) < N$. Now assume that for all $k \in \bN_+$ there exists
a non-zero $W_k \in W^{1, 2}(M, TM)$ such that
\[
 \frac{1}{2} \int_M \frac{\left|\bL W\right|^2}{2\eta} d\mu^g \leq \frac{1}{k} \left\|W_k\right\|^2_{W^{1, 2}(M, TM)}.
\]
Without loss of generality, we can assume that $\left\|W_k\right\|_{L^q(M, TM)} = 1$.
Note that, due to the Sobolev embedding, the norm
\[
 \left\|W\right\|^2 \definedas \int_M \left|\nabla W\right|^2 d\mu^g + \left\|W\right\|^2_{L^q(M, TM)} 
\]
is equivalent to the usual Sobolev norm on $W^{1, 2}(M, TM)$, so for some constant $\delta > 0$,
we have $\|W\|_{W^{1, 2}(M, TM)}^2 \leq \delta \left\|W\right\|^2$ for all $W \in W^{1, 2}(M, TM)$.
Hence, from the previous estimates
\begin{align*}
&\frac{1}{\Lambda} \left[\int_M \left|\nabla W_k\right|^2d\mu^g - \left\|\ric\right\|_{L^{p/2}(M, S_2M)} \left\|W_k\right\|^2_{L^q(M, TM)}\right]\\
&\qquad \leq \frac{\delta}{k} \left(\int_M \left|\nabla W_k\right|^2 d\mu^g + \left\|W_k\right\|^2_{L^q(M, TM)}\right).
\end{align*}
Equivalently, using $\|W_k\|_{L^q(M, TM)} = 1$,
\[
\left(\frac{1}{\Lambda} - \frac{\delta}{k}\right) \int_M \left|\nabla W_k\right|^2d\mu^g \leq  \frac{\left\|\ric\right\|_{L^{p/2}(M, S_2M)}}{\Lambda} + \frac{\delta}{k}.
\]
It follows that $(W_k)$ is bounded in $W^{1, 2}(M, TM)$. Since $W^{1, 2}(M, TM)$ compactly embeds into $L^q(M, TM)$,
we can assume that $W_k$ converges to $W_\infty$ strongly in $L^q(M, TM)$ and weakly in $W^{1, 2}(M, TM)$.
We have $\|W_\infty\|_{L^q(M, TM)} = 1$ so $W_\infty$ is non-zero but
\begin{align*}
\int_M \frac{\left|\bL W_\infty\right|^2}{2\eta} d\mu^g
 &= \lim_{k \to \infty} \int_M \frac{\left\<\bL W_\infty, \bL W_k\right\>}{2\eta} d\mu^g\\
 &\leq \left(\int_M \frac{\left|\bL W_\infty\right|^2}{2\eta} d\mu^g\right)^{1/2} \left(\lim_{k \to \infty} \int_M \frac{\left|\bL W_k\right|^2}{2\eta} d\mu^g\right)^{1/2} = 0.
\end{align*}
Namely, we obtained a contradiction.

It then follows from the Lax-Milgram theorem that the functional $J$ admits
a unique minimizer $W \in W^{1, 2}(M, TM)$ which is then a weak solution to
\eqref{eqVectLapl}. Elliptic regularity then implies that $W \in W^{2, q}(M, TM)$.
\end{proof}

%%%%%%%%%%%%%%%%%%%%%%%%%%%%%%%%%%%%%%%%%%%%%%%%%%%%%%%%%%%%%%%%%%%%%%%%%
\section{An implicit function argument}\label{secImplicit}
%%%%%%%%%%%%%%%%%%%%%%%%%%%%%%%%%%%%%%%%%%%%%%%%%%%%%%%%%%%%%%%%%%%%%%%%%
In this section, we make the following regularity assumptions:
\begin{itemize}
 \item $g \in W^{2, p/2}(M, S_2M)$,
 \item $\tau \in W^{1, p/2}(M, \bR)$,
 \item $\sigma \in L^p(M, S_2M)$,
 \item $\eta \in W^{1, p}(M, \bR)$, $\eta > 0$
\end{itemize}
for some $p > n$.
The idea in \cite{GicquaudNgo} is to introduce a parameter $\lambda > 0$
in the system \eqref{system}. Namely, we set $\phi = \lambda \phitil$
and $W = \lambda^N \Wtil$ so the system becomes

\begin{subequations}\label{systemQ}
\begin{empheq}[left=\empheqlbrace]{align}
-\frac{4(n-1)}{n-2} \Delta_g \phitil + \scal \phitil &= - \frac{n-1}{n} \lambda^{N-2} \tau^2 \phitil^{N-1}\nonumber\\
 &\qquad + \left| \lambda^{-\frac{N+2}{2}} \sigma + \frac{\lambda^{\frac{N-2}{2}}}{2\eta}\bL_g \Wtil \right|_g^2 \phitil^{-N-1}, \label{eqLichnerowiczQ}\\
\DeltaL \Wtil &= \frac{n-1}{n} \phitil^N \nabla\tau. \label{eqVectorQ}
\end{empheq}
\end{subequations}
(note that the rescaling we present here is different from the one in
\cite{GicquaudNgo}). Setting
\begin{equation}\label{eqScalingSigma0}
 \sigma = \lambda^{\frac{N+2}{2}} \sigmatil
\end{equation}
the Lichnerowicz equation \eqref{eqLichnerowiczQ} reads
\[
 -\frac{4(n-1)}{n-2} \Delta_g \phitil + \scal \phitil = - \frac{n-1}{n} \lambda^{N-2} \tau^2 \phitil^{N-1}
 + \left| \sigmatil + \frac{\lambda^{\frac{N-2}{2}}}{2\eta}\bL_g \Wtil \right|_g^2 \phitil^{-N-1}.
\]
Letting $\lambda$ go to zero, we see that the system \eqref{systemQ}
is a perturbation of
\begin{subequations}\label{systemR0}
\begin{empheq}[left=\empheqlbrace]{align}
-\frac{4(n-1)}{n-2} \Delta_g \phitil_0 + \scal \phitil_0 &= \big|\sigmatil\big|_g^2 \phitil_0^{-N-1}, \label{eqLichnerowiczR0}\\
\DeltaL \Wtil_0 &= \frac{n-1}{n} \phitil_0^N \nabla\tau. \label{eqVectorR0}
\end{empheq}
\end{subequations}
So $\Wtil$ has disappeared from Equation \eqref{eqLichnerowiczR0}.
Solving the equation \eqref{eqLichnerowiczR0} requires that the Yamabe
invariant of $(M, g)$ be positive since the metric $\gbar = \phitil_0^{N-2} g$
has scalar curvature $\scaltil_0 = |\sigma|_g^2 \phitil_0^{-2N}$, which is
non-negative and non-zero. This explains why the method was limited to $\cY_g > 0$.\\

As we indicated before, in the case $\cY_g = 0$, the conformal Laplacian
\[
 u \mapsto -\frac{4(n-1)}{n-2} \Delta_g u + \scal u
\]
has a 1-dimensional kernel generated by a positive function
$\phi_0 \in W^{2, p/2}(M, \bR)$ which we normalize so that
\[
 \int_M \phi_0^2 d\mu^g = 1.
\]
Since the conformal Laplacian is Fredholm with index zero and formally
self adjoint, the equation
\[
 -\frac{4(n-1)}{n-2} \Delta_g u + \scal u = f
\]
with $f \in L^{p/2}(M, \bR)$ is solvable iff
\[
 \int_M f \phi_0 d\mu^g = 0.
\]
The solution $u \in W^{2, p/2}(M, \bR)$ is unique up to the addition of a constant
multiple of $\phi_0$ so it is unique if we impose further that
\[
 \int_M u \phi_0 d\mu^g = 0.
\]

If we change the scaling law \eqref{eqScalingSigma0} of $\sigma$ to
\begin{equation}\label{eqScalingSigma}
 \sigma = \lambda^{N} \sigmatil,
\end{equation}
the system \eqref{system} reads now
\[
\left\lbrace
\begin{aligned}
-\frac{4(n-1)}{n-2} \Delta_g \phitil + \scal \phitil &= \lambda^{N-2}\left( - \frac{n-1}{n}  \tau^2 \phitil^{N-1}
   + \left| \sigmatil + \frac{1}{2\eta}\bL_g \Wtil \right|_g^2 \phitil^{-N-1}\right),\\
\DeltaL \Wtil &= \frac{n-1}{n} \phitil^N \nabla\tau.
\end{aligned}
\right.
\]
Hence, setting
\begin{equation}\label{eqPhitil}
 \phitil \definedas c_\lambda \phi_0 + \lambda^{N-2} \psi_\lambda,
\end{equation}
where $\psi_\lambda$ belongs to the space $\Wring^{2, p/2}(M, \bR)$ of $W^{2, p}(M, \bR)$-functions orthogonal
to $\phi_0$ for the $L^2$-product:
\begin{equation}\label{eqDefRing}
 \Wring^{2, p/2}(M, \bR) \definedas \left\{u \in W^{2, p/2}(M, \bR), \int_M u \phi_0 d\mu^g = 0\right\}
\end{equation}
(more generally, for any function space $F$ such that $F \hookrightarrow L^2(M, \bR)$,
we will denote by $\mathring{F}$ the set of functions $u$ belonging to $F$ that
are $L^2$-orthogonal to $\phi_0$).
We finally arrive at
\[
\left\lbrace
\begin{aligned}
 -\frac{4(n-1)}{n-2} \Delta_g \psi_\lambda + \scal \psi_\lambda &= \left( - \frac{n-1}{n}  \tau^2 \phitil^{N-1}
   + \left| \sigmatil + \frac{1}{2\eta}\bL_g \Wtil \right|_g^2 \phitil^{-N-1}\right),\\
\DeltaL \Wtil &= \frac{n-1}{n} \phitil^N \nabla\tau.
\end{aligned}
\right.
\]
The role of the constant $c_\lambda$
(which still appears implicitly in the definition of $\phitil$)
will be to ensure that the right hand
side of the rescaled Lichnerowicz equation is $L^2$-orthogonal to $\phi_0$.
To emphasize this, we rewrite the system as follows:
\begin{subequations}\label{systemS}
\begin{empheq}[left=\empheqlbrace]{align}
-\frac{4(n-1)}{n-2} \Delta_g \psi_\lambda + \scal \psi_\lambda &= - \frac{n-1}{n}  \tau^2 \phitil^{N-1}
   + \left| \sigmatil + \frac{\bL_g \Wtil}{2\eta} \right|_g^2 \phitil^{-N-1}, \label{eqLichnerowiczS}\\
\frac{n-1}{n} \int_M \tau^2 \phi_0 \phitil^{N-1} d\mu^g &= \int_M \phi_0 \left| \sigmatil + \frac{\bL_g \Wtil}{2\eta} \right|_g^2 \phitil^{-N-1} d\mu^g,\label{eqOrthoS}\\
\DeltaL \Wtil &= \frac{n-1}{n} \phitil^N \nabla\tau. \label{eqVectorS}
\end{empheq}
\end{subequations}

%%%%%%%%%%%%%%%%%%%%%%%%%%%%%%%%%%%%%%%%%%%%%%%%%%%%%%%%%%%%%%%%%%%%%%%%%
\subsection{The limit \texorpdfstring{$\lambda = 0$}{TEXT}}
%%%%%%%%%%%%%%%%%%%%%%%%%%%%%%%%%%%%%%%%%%%%%%%%%%%%%%%%%%%%%%%%%%%%%%%%%
In the limit $\lambda = 0$, we have $\phitil = c_0 \phi_0$ so the system \eqref{systemS}
becomes
\begin{subequations}\label{systemS0}
\begin{empheq}[left=\empheqlbrace]{align}
-\frac{4(n-1)}{n-2} \Delta_g \psi_0 + \scal \psi_0 &= - \frac{n-1}{n}  \tau^2 \phitil^{N-1}
   + \left| \sigmatil + \frac{\bL_g \Wtil}{2\eta} \right|_g^2 \phitil^{-N-1}, \label{eqLichnerowiczS0}\\
\frac{n-1}{n} c_0^{2N} \int_M \tau^2 \phi_0^N d\mu^g &= \int_M \phi_0^{-N} \big| \sigmatil + \bL_g \Wtil_0 \big|_g^2 d\mu^g,\label{eqOrthoS0}\\
\DeltaL \Wtil_0 &= \frac{n-1}{n} c_0^N \phi_0^N \nabla\tau. \label{eqVectorS0}
\end{empheq}
\end{subequations}
We solve this system from bottom to top. Namely, from standard arguments,
there exists a unique solution $\Wbar \in W^{2, p/2}(M, TM)$ to
\begin{equation}\label{eqDefWbar}
 \DeltaL \Wbar = \frac{n-1}{n} \phi_0^N \nabla\tau.
\end{equation}
so we have $\Wtil_0 = c_0^N \Wbar$. Inserting it into Equation \eqref{eqOrthoS0},
we find 
\begin{equation}\label{eqC}
 \frac{n-1}{n} c_0^{2N} \int_M \tau^2 \phi_0^N d\mu^g
 = \int_M \left(|\sigmatil|^2 + \frac{1}{\eta} c_0^N \<\sigmatil, \bL\Wbar\>
    + c_0^{2N} \frac{|\bL_g \Wbar|_g^2}{4\eta^2}\right) \phi_0^{-N} d\mu^g.
\end{equation}
This is a second order equation in $c_0^N$ which we have to assume has a
positive solution. From Descartes' rule of signs (see e.g. \cite{Yap}),
Equation \eqref{eqC} has a unique positive solution provided that
\begin{equation}\label{eqCondition}
 \frac{n-1}{n} \int_M \tau^2 \phi_0^N d\mu^g > \int_M \frac{\left|\bL_g \Wbar\right|_g^2}{4\eta^2} \phi_0^{-N} d\mu^g.
\end{equation}
Note however that there might exist situations in which Equation \eqref{eqC}
has two positive solutions. We plan to investigate this question later.

Having fulfilled the last two equations, we can finally solve Equation
\eqref{eqLichnerowiczS0} for $\psi_0 \in \Wring^{2, p/2}(M, \bR)$. We
have thus proven

\begin{proposition}\label{propVanishingLambda}
Under the assumption \eqref{eqCondition}, there exist a unique solution
$(c_0, \psi_0, \Wtil_0) \in \bR_+ \times \Wring^{2, p/2}(M, \bR) \times W^{2, p/2}(M, TM)$ to the
system \eqref{systemS0} where $\Wring^{2, p/2}(M, \bR)$ is defined in
\eqref{eqDefRing}.
\end{proposition}

%%%%%%%%%%%%%%%%%%%%%%%%%%%%%%%%%%%%%%%%%%%%%%%%%%%%%%%%%%%%%%%%%%%%%%%%%
\subsection{Extending to \texorpdfstring{$\lambda > 0$}{TEXT}}
%%%%%%%%%%%%%%%%%%%%%%%%%%%%%%%%%%%%%%%%%%%%%%%%%%%%%%%%%%%%%%%%%%%%%%%%%
Let $\Omega \subset \bR \times \bR \times \Wring^{2, p/2}$ be the following
open subset:
\[
 \Omega \definedas \{(\lambda, c, \psi) \in \bR \times \bR \times \Wring^{2, p/2}, \text{ s.t. } c \phi_0 + \lambda^{N-2} \psi > 0\}
\]
We define the operator
\[
\Phi: \Omega \times  W^{2, p/2}(M, TM) \to \mathring{L}^2 \times \bR \times L^p(M, TM)
\]
as follows:
\begin{equation}\label{eqDefPhi}
 \Phi_\lambda(c, \psi, W) \definedas
\begin{pmatrix}
 \Pi\left(-\frac{4(n-1)}{n-2} \Delta_g \psi + \scal \psi + \frac{n-1}{n}  \tau^2 \phitil^{N-1} -  \big| \sigmatil + \bL_g \Wtil \big|_g^2 \phitil^{-N-1}\right)\\
 \frac{n-1}{n} \int_M \tau^2 \phi_0 \phitil^{N-1} d\mu^g - \int_M \phi_0 \big| \sigmatil + \bL_g \Wtil \big|_g^2 \phitil^{-N-1} d\mu^g\\
 \DeltaL \Wtil - \frac{n-1}{n} \phitil^N \nabla\tau
\end{pmatrix},
\end{equation}
where we used $\phitil = c \phi_0 + \lambda^{N-2} \psi$ as a shorthand (see
\eqref{eqPhitil}) and where $\Pi$ denotes the $L^2$-orthogonal projection onto
$\mathring{L}^2(M, \bR)$:
\[
 \Pi(f) = f - \left(\int_M f \phi_0 d\mu^g\right) \phi_0.
\]
Solving the system \eqref{systemS} is then equivalent to finding solutions
to
\[
 \Phi_\lambda(c, \psi, \Wtil) = \begin{pmatrix}0\\ 0\\ 0\end{pmatrix}.
\]
It is routine to check that $\Phi$ is well defined and $C^1$.
To apply the implicit function theorem, we only need to check that the differential
of $\Phi_\lambda$ with $\lambda = 0$ kept fixed is invertible at the point $(c_0, \psi_0, \Wtil_0)$.
Since $\phitil = c \phi_0$ when $\lambda = 0$, $\Phi_0$ reads
\[
 \Phi_0(c, \psi, W) =
\begin{pmatrix}
 \Pi\left(-\frac{4(n-1)}{n-2} \Delta_g \psi + \scal \psi + \frac{n-1}{n} c^{N-1} \tau^2 \phi_0^{N-1} -  \big| \sigmatil + \bL_g \Wtil \big|_g^2 c^{-N-1} \phi_0^{-N-1}\right)\\
 \frac{n-1}{n} c^{N-1} \int_M \tau^2 \phi_0^N d\mu^g - c^{-N-1} \int_M \big| \sigmatil + \bL_g \Wtil \big|_g^2 \phi_0^{-N} d\mu^g\\
 \DeltaL \Wtil - \frac{n-1}{n} c^N \phi_0^N \nabla\tau
\end{pmatrix},
\]
Its differential at $(c_0, \psi_0, \Wtil_0)$ can be computed:
\begin{equation}\label{eqDiffPhi}
 D \Phi_0(c_0, \psi_0, W)\begin{pmatrix} \psi'\\ c'\\W'\end{pmatrix} =
\begin{pmatrix}
 -\frac{4(n-1)}{n-2} \Delta_g + \scal & \Pi(F) & \Pi(\ell(\cdot))\\
 0 & \int_M \phi_0 Fd\mu^g & \int_M \phi_0\ell(\cdot) d\mu^g \\
 0 & -N \frac{n-1}{n} c^{N-1} \phi_0^N \nabla\tau & \DeltaL
\end{pmatrix}
\begin{pmatrix} \psi'\\ c'\\W'\end{pmatrix}
\end{equation}
where we used the following notations:
\[
\left\lbrace
\begin{aligned}
 F &\definedas \frac{n-1}{n} (N-1) c^{N-2} \tau^2 \phi_0^{N-1} + (N+1) \big| \sigmatil + \bL_g \Wtil \big|_g^2 \phi_0^{-N-1} c^{-N-2},\\
 \ell(W') &\definedas -2 c^{-N-1} \phi_0^{-N-1} \big\<\sigmatil + \bL_g \Wtil, \bL_g W'\big\>_g.
\end{aligned}
\right.
\]
The matrix of the differential in \eqref{eqDiffPhi} is not upper triangular
as is the case in \cite{GicquaudNgo}. However, the conformal Laplacian
appearing in the upper left corner of the matrix is an isomorphism from
$\Wring^{2, p/2}(M, \bR)$ to $\Lring^{p/2}(M, \bR)$ so it suffices to check
that the lower $2 \times 2$ block
is invertible. We show that for any $d \in \bR$ and any $V \in L^p(M, TM)$
there exists a unique solution to the system
\begin{equation}\label{eqSystem}
\left\lbrace
\begin{aligned}
 d &= c' \int_M \phi_0 Fd\mu^g + \int_M \phi_0\ell(W') d\mu^g,\\
 V &=  - c' N \frac{n-1}{n} c_0^{N-1} \phi_0^N \nabla\tau + \DeltaL W'.
\end{aligned}
\right.
\end{equation}
The second equation can be solved explicitely for $W'$:
\[
 W' = N c_0^{N-1} c' \Wbar + \left(\DeltaL\right)^{-1} V,
\]
where $(\DeltaL)^{-1}: L^{p/2}(M, TM) \to W^{2, p/2}(M, TM)$ denotes
the inverse of $\DeltaL$.
Injecting into the first equation, we get
\[
 c' \int_M \phi_0 Fd\mu^g - 2 c_0^{-N-1} \int_M \phi_0^{-N-1} \big\<\sigmatil + \bL_g \Wtil, \bL_g \left(N c_0^{N-1} c' \Wbar + \left(\DeltaL\right)^{-1} V\right)\big\>_g d\mu^g = d.
\]
This equation can be solved for $c'$ if and only if
\[
 \left(\frac{n-1}{n} \int_M \tau^2 \phi_0^N d\mu^g - \int_M |\bL_g \Wbar|^2 d\mu^g\right) c_0^N \neq \int_M \big\<\sigmatil, \bL_g \Wtil\big\> \phi_0^{-N} d\mu^g
\]
where we used Equation \eqref{eqC}. This condition is to be expected,
it means that $c_0^N$ is not a double root of Equation \eqref{eqC} seen as a second
order equation in $c_0^N$. We have thus proven the following proposition:

\begin{proposition}\label{propSmallLambda}
Under the assumption \eqref{eqCondition}, there exist $\lambda_0 > 0$ and
a continuous curve of solutions $(c_\lambda, \psi_\lambda, \Wtil_\lambda)$
to the system \eqref{systemS}.
\end{proposition}

Using now the rescaling presented at the beginning of the section, we have
proven the first theorem of this paper:

\begin{theorem}\label{thmImplicit}
 Given $(M, g, \tau, \eta)$ and $\sigmatil \in L^p(M, S_2M)$, $\sigmatil \neq 0$,
where the regularities are indicated at the beginning of the section,
and assuming further that
\[
 \frac{n-1}{n} \int_M \tau^2 \phi_0^N d\mu^g > \int_M \frac{|\bL_g \Wbar|_g^2}{4\eta^2} \phi_0^{-N} d\mu^g,
\]
 there exists a $\lambda_0 > 0$ such that the system \eqref{system}
with $\sigma \equiv \lambda \sigmatil$ has at least a solution $(\phi, W) \in W^{2, p/2}(M, \bR) \times W^{2, p/2}(M, TM)$
for all $\lambda \in (0, \lambda_0)$.
\end{theorem}

%%%%%%%%%%%%%%%%%%%%%%%%%%%%%%%%%%%%%%%%%%%%%%%%%%%%%%%%%%%%%%%%%%%%%%%%%
\section{A small TT-tensor argument}\label{secSmallTT}
%%%%%%%%%%%%%%%%%%%%%%%%%%%%%%%%%%%%%%%%%%%%%%%%%%%%%%%%%%%%%%%%%%%%%%%%%
In this section, we will require stronger regularity for $\tau$ than in
the previous section. Namely, we assume that, for some $p > n$,
\begin{itemize}
 \item $g \in W^{2, p/2}(M, S_2M)$,
 \item $\eta \in W^{1, p}(M, \bR)$, $\eta > 0$,
 \item $\sigma \in L^p(M, S_2M)$,
 \item $\tau \in W^{1, t}(M, \bR)$, where $t > t_0$ with
\begin{equation}\label{eqDefT0}
 t_0 = \frac{2n(n-1)}{3n-2}.
\end{equation}
\end{itemize}
The reason why we have to impose stronger regularity for $\tau$ will become
apparant in the course of the proof. We also take advantage of the fact that
the CTS method is conformally covariant (see \cite{MaxwellConformalMethod})
to enforce the condition $\scal \equiv 0$. In particular
$\phi_0 \equiv 1$. We will also assume that $(M, g)$ has volume one:
\[
 \vol(M, g) = 1.
\]
So if $p \leq q$,
\[
 \|f\|_{L^p(M, \bR)} \leq \|f\|_{L^q(M, \bR)}
\]
for any measurable function. This can be achieved by rescaling the metric
by some (constant) factor.

We prove an analog of the result in \cite{Nguyen, GicquaudNguyen},
namely the existence of a solution to the system \eqref{system} when
$\sigma$ is small in $L^2(M, \bR)$.
To keep expressions short, we adopt at some points a probabilistic
notation and denote
\[
 \bE[f] \definedas \int_M f d\mu^g
\]
and
\[
 \bE_\tau[f] \definedas \frac{1}{\bE[\tau^2]} \bE[\tau^2 f] = \frac{1}{\int_M \tau^2 d\mu^g}\int_M \tau^2 f d\mu^g
\]
for any function $f$ for which this makes sense (e.g. $f \in L^{\frac{2t}{2t-1}}$).
The strategy is similar to the previous ones in \cite{HNT1,HNT2,MaxwellNonCMC,Nguyen,GicquaudNguyen}.
The previous section suggests that we have to decompose
$\phi$ as $\phi = c + \mathring{\phi}$ where $c$ is a constant and $\mathring{\phi}$
has zero average. Yet, for technical reasons, it appears more
interesting not to decompose $\phi$ itself but $\phi^N$ and to
do it in a way that involves $\tau^2$.
For any function space $X$, we set
\[
 \breve{X} \definedas \{f \in X, \bE_\tau[f] = 0\}.
\]
Let $p_0$ be defined as follows:
\begin{equation}\label{eqDefP0}
 \frac{1}{p_0} = \frac{2}{p} - \frac{1}{t}
\end{equation}
Given $c_{\max} > 0$ and $r > 0$ to be chosen later, we let
$C_0 \subset L^{p_0}(M, \bR)$ be the following subset
\begin{equation}\label{eqDefC}
 C_0 \definedas \{u \in L^{p_0}(M, \bR),~u \geq 0,\bE_\tau[u] \leq c_{\max} \text{ and } \|u - \bE_\tau[u]\|_{L^{\frac{N}{2}+1}(M, \bR)} \leq r\}.
\end{equation}
The reason why we work in the Lebesgue $L^{\frac{N}{2}+1}$-norm will become
apparant later. We construct a mapping $\Psi: C_0 \to L^{p_0}(M, \bR)$
as follows:
\begin{enumerate}
 \item Given $u = c + \psi \in C_0$ ($c \in \bR$ and $\psi \in \breve{L}^{p_0}(M, \bR)$),
we let $W \in W^{2, p/2}(M, TM)$ be the unique solution to the following equation
\begin{equation}\label{eqDefPsi1}
 \DeltaL W = \frac{n-1}{n} u \nabla\tau,
\end{equation}
see Proposition \ref{propVector}. Note that $p_0$ is chosen so that the
right hand side belongs to $L^{p/2}(M, TM)$.
So, according to the notation introduced in \eqref{eqDefWbar}, we have
\[
 W = c \Wbar + W_\psi,
\]
where $W_\psi$ denotes the solution to \eqref{eqDefPsi1} with $u$ replaced by $\psi$.
 \item We next solve the Lichnerowicz equation for $\phi > 0$ with the $W$
we found in the first step. The solution $\phi \in W^{2, p/2}(M, \bR)$ ($\subset L^\infty(M, \bR)$)
is known to exist from Proposition \ref{propLichnerowicz}.
 \item Finally, we set $\Psi(u) = \phi^N = c' + \psi'$ where $c' \in \bR$
and $\psi' \in \breve{L}^{p_0}(M, \bR)$.
\end{enumerate}
$\Psi$ is the composition of three continuous mappings and, hence,
continuous and its fixed points correspond to solutions of the system
\eqref{system}. Our first aim is to show that, if $\|\sigma\|_{L^2(M, S_2M)}$ is small
enough, we can adjust $c_{\max}$ and $r$ so that $\Psi(C_0) \subset C_0$.

To estimate $\phi$, we multiply the Lichnerowicz equation \eqref{eqLichnerowicz} by $\phi^{N+1}$
and integrate over $M$:

\[
 -\frac{4(n-1)}{n-2} \int_M \phi^{N+1} \Delta \phi d\mu^g + \frac{n-1}{n} \int_M \tau^2 \phi^{2N} d\mu^g = \int_M \left|\sigma + \frac{\bL W}{2 \eta}\right|^2 d\mu^g.
\]
Integrating by parts the first term, we have
\begin{align*}
-\frac{4(n-1)}{n-2} \int_M \phi^{N+1} \Delta \phi d\mu^g
 &= \frac{4(n-1)}{n-2} \int_M \left\<d\phi^{N+1}, d\phi\right\> d\mu^g\\
 &= \frac{4(n-1)}{n-2} (N+1) \int_M \left\<\phi^N d\phi, d\phi\right\> d\mu^g\\
 &= \frac{4(n-1)}{n-2} (N+1) \int_M \left\<\phi^{N/2} d\phi, \phi^{N/2}d\phi\right\> d\mu^g\\
 &= \frac{4(n-1)}{n-2} \frac{N+1}{\left(\frac{N}{2}+1\right)^2} \int_M \left|d\left(\phi^{\frac{N}{2}+1}\right)\right|^2 d\mu^g\\
 &= \frac{3n-2}{n-1} \int_M \left|d\left(\phi^{\frac{N}{2}+1}\right)\right|^2 d\mu^g.
\end{align*}
So the estimate for $\phi$ reads
\begin{equation}\label{eqEstimateLich}
 \frac{3n-2}{n-1} \int_M \left|d\left(\phi^{\frac{N}{2}+1}\right)\right|^2 d\mu^g + \frac{n-1}{n} \int_M \tau^2 \phi^{2N} d\mu^g = \int_M \left|\sigma + \frac{\bL W}{2 \eta}\right|^2 d\mu^g.
\end{equation}

We prove the following variant of the Sobolev inequality:

\begin{lemma}\label{lmSobolev}
 There exists a constant $s = s(M, g, \tau)> 0$ such that for any function $f \in W^{1, 2}(M, \bR)$
we have
\begin{equation}\label{eqSobolev}
 \left\|f - \bE_\tau[f]\right\|^2_{L^N(M, \bR)} \leq s \left\|df\right\|_{L^2(M, \bR)}^2.
\end{equation}
\end{lemma}

\begin{proof}
We argue by contradiction and assume that there exists no constant $s > 0$
such that Inequality \eqref{eqSobolev} holds for all $f$. There exists a sequence
$(f_k)_{k \in \bN}$ such that
\[
 \left\|f_k - \bE_\tau[f_k]\right\|^2_{L^N(M, \bR)} \geq k \left\|df_k\right\|_{L^2(M, \bR)}^2
\]
for all $k$. From the Sobolev embedding theorem, there exists $s_0 > 0$ such
that
\[
 \left\|f_k - \bE_\tau[f_k]\right\|^2_{L^N(M, \bR)} \leq s_0 \left( \left\|df_k\right\|_{L^2(M, \bR)}^2 + \left\|f_k - \bE_\tau[f_k]\right\|^2_{L^2(M, \bR)}\right).
\]
As a consequence, we have that
\[
  \left\|f_k - \bE_\tau[f_k]\right\|^2_{L^2(M, \bR)} \geq \frac{k - s_0}{s_0} \left\|df_k\right\|_{L^2(M, \bR)}^2.
\]
By replacing $f_k$ by $f_k - \bE_\tau[f_k]$, we can assume that
$\bE_\tau[f_k] = 0$ for all $k$, and, by rescaling $f_k$, we impose that
$\|f_k\|_{L^2(M, \bR)} = 1$ for all $k$.
We finally also assume that there exists $f_\infty \in W^{1, 2}(M, \bR)$
such that $(f_k)$ converges to $f_\infty$ weakly in $W^{1, 2}(M, \bR)$ and strongly
in $L^2(M, \bR)$. This comes from the fact that $W^{1, 2}(M, \bR)$ is a Hilbert space
(hence reflexive) and that the injection $W^{1, 2}(M, \bR) \into L^2(M, \bR)$
is compact. In particular, we have that, for any $u \in W^{1, 2}(M, \bR)$,
\[
 \left|\int_M \<du, df_k\> d\mu^g\right| \leq \|u\|_{W^{1, 2}(M, \bR)} \|df_k\|_{L^2(M, \bR)} \leq \frac{s_0}{k-s_0}\|u\|_{W^{1, 2}(M, \bR)} \to_{k \to \infty} 0.
\]
Thus,
\[
 \int_M \<du, df_\infty\> d\mu^g = 0
\]
for all $u \in W^{1, 2}(M, \bR)$. Namely, $f_\infty$ is a weak solution to $\Delta f_\infty = 0$
so $f_\infty$ is a constant. Since $\|f_k\|_{L^2(M, \bR)} = 1$ for all $k$ and $f_k \to f_\infty$
strongly in $L^2(M, \bR)$, we have $\|f_\infty\|_{L^2(M, \bR)} = 1$. So $f_\infty \equiv \pm 1$ is
a non-zero constant function. Yet,
\[
 f_\infty = \bE_\tau[f_\infty] = \lim_{k \to \infty} \bE_\tau[f_k] = 0,
\]
a contradiction. This proves the lemma.
\end{proof}
Applying eagerly the previous lemma to Estimate \eqref{eqEstimateLich}, we get:
\begin{equation}\label{eqEstimateLich2}
 \frac{3n-2}{n-1} \frac{1}{s} \left\|\phi^{\frac{N}{2}+1} - \bE_\tau\left[\phi^{\frac{N}{2}+1}\right]\right\|_{L^N(M, \bR)}^2 + \frac{n-1}{n} \int_M \tau^2 \phi^{2N} d\mu^g \leq \int_M \left|\sigma + \frac{\bL W}{2 \eta}\right|^2 d\mu^g.
\end{equation}
But, as we indicated at the beginning of the section, we want to decompose
$\phi^N$ not $\phi^{\frac{N}{2}+1}$! So we need a second lemma:

\begin{lemma}\label{lmMoment}
For any $\beta > 1$, any $\alpha \in (1, 2)$ and any positive function $f$ we have
\begin{equation}\label{eqMoment}
\begin{aligned}
\left\|f^\alpha - \bE_\tau[f^\alpha]\right\|_{L^\beta(M, \bR)}
 &\leq \alpha \bE_\tau[f^\alpha]^{\frac{\alpha-1}{\alpha}} \left[1+ \frac{\|\tau\|_{L^{2\gamma}(M, \bR)}^{2/\alpha}}{\|\tau\|_{L^2(M, \bR)}^{2/\alpha}}\right] \left\|f - \bE_\tau[f]\right\|_{L^{\alpha\beta}(M, \bR)}\\
 & \qquad + \left[1+\frac{\|\tau\|_{L^{2\gamma}(M, \bR)}^2}{\|\tau\|_{L^2(M, \bR)}^2}\right] \left\|f - \bE_\tau[f]\right\|^\alpha_{L^{\alpha\beta}(M, \bR)},
\end{aligned}
\end{equation}
where $\gamma$ satisfies
\[
 \frac{1}{\beta}+\frac{1}{\gamma} = 1.
\]
\end{lemma}

\begin{proof}
Before getting into the proof of the lemma, we state and prove the following inequality:
\begin{equation}\label{eqCrap}
 \forall x \in \bR_+,~|x^\alpha - 1| \leq \alpha |x - 1| + |x-1|^\alpha.
\end{equation}
First assume that $x \in (0, 1)$. Then, since $\alpha > 1$, we have,
\[
 |x^\alpha - 1| = 1-x^\alpha
= \alpha \int_x^1 y^{\alpha-1} dy \leq \alpha \int_x^1 dy = \alpha (1-x)
\leq \alpha |x - 1| + |x-1|^\alpha.
\]
Next, for $x > 1$, the function
\[
 h:y \mapsto \frac{y^{\alpha-1}-1}{(y-1)^{\alpha-1}}
\]
has derivative
\[
 h'(y) = \frac{\alpha-1}{(x-1)^\alpha} \left(1-x^{\alpha-2}\right)
\]
so, since $\alpha \in (1, 2)$, $h$ is increasing on the interval $(1, \infty)$
and tends to $1$ at infinity. Thus, $h(y) \leq 1$ for all $y \in (1, \infty)$.
This inequality can be rewritten
\[
 \alpha y^{\alpha-1} \leq \alpha + \alpha (y-1)^{\alpha-1}.
\]
Integrating from $y=1$ to $y=x$, we obtain Inequality \eqref{eqCrap} for
all $x > 1$. This concludes the proof of Inequality \eqref{eqCrap}.

Assume now that $a, b \in \bR_+$. We set $x = a/b$ in \eqref{eqCrap}
and multiply it by $b^\alpha$. We obtain the homogeneous form of \eqref{eqCrap}:
\begin{equation}\label{eqCrap2}
 \forall a, b \in \bR_+,~|a^\alpha - b^\alpha| \leq \alpha b^{\alpha-1}|a - b| + |a-b|^\alpha.
\end{equation}

We now prove Inequality \eqref{eqMoment}. We first compare $\bE_\tau[f^\alpha]$
with $\bE_\tau[f]^\alpha$. It follows from Jensen's inequality that
\[
 \bE_\tau[f]^\alpha \leq \bE_\tau[f^\alpha].
\]
For the opposite direction, we use Minkowski's inequality:
\begin{align*}
\bE_\tau[f^\alpha]^{1/\alpha}
 &= \bE_\tau[\left|f - \bE_\tau[f] + \bE_\tau[f]\right|^\alpha]^{1/\alpha}\\
 &\leq \bE_\tau[\left|f - \bE_\tau[f]\right|^\alpha]^{1/\alpha} + \bE_\tau[\bE_\tau[f]^\alpha]^{1/\alpha}\\
 &\leq \bE_\tau[\left|f - \bE_\tau[f]\right|^\alpha]^{1/\alpha} + \bE_\tau[f]\\
 &\leq \bE_\tau[\left|f - \bE_\tau[f]\right|^\alpha]^{1/\alpha} + \bE_\tau[f].
\end{align*}
We then use H\"older's inequality:
\begin{align*}
\bE_\tau[\left|f - \bE_\tau[f]\right|^\alpha]^{1/\alpha}
 &= \frac{1}{\bE[\tau^2]^{1/\alpha}} \bE[\tau^2 \left|f - \bE_\tau[f]\right|^\alpha]^{1/\alpha}\\
 &\leq \frac{1}{\bE[\tau^2]^{1/\alpha}} \bE[\tau^{2\gamma}]^{1/(\alpha\gamma)} \bE[\left|f - \bE_\tau[f]\right|^{\alpha\beta}]^{1/(\alpha\beta)}.
\end{align*}
As a consequence, we have proven
\[
 \bE_\tau[f]^\alpha \leq \bE_\tau[f^\alpha] \leq \left(\bE_\tau[f] + \frac{\bE[\tau^{2\gamma}]^{1/(\alpha\gamma)}}{\bE[\tau^2]^{1/\alpha}}  \bE[\left|f - \bE_\tau[f]\right|^{\alpha\beta}]^{1/(\alpha\beta)}\right)^\alpha.
\]
In particular,
\[
 \left|\bE_\tau[f^\alpha]^{1/\alpha} - \bE_\tau[f]\right| \leq \frac{\bE[\tau^{2\gamma}]^{1/(\alpha\gamma)}}{\bE[\tau^2]^{1/\alpha}} \left\|f - \bE_\tau[f]\right\|_{L^{\alpha\beta}(M, \bR)}.
\]
We apply Inequality \eqref{eqCrap2} to $u=\bE_\tau[f^\alpha]^{1/\alpha}$
and $v = \bE_\tau[f]$. From the previous inequality, we infer
\begin{equation}\label{eqMean}
\begin{aligned}
\left|\bE_\tau[f^\alpha] - \bE_\tau[f]^\alpha\right|
 &\leq \alpha \left(\bE_\tau[f^\alpha]\right)^{\frac{\alpha-1}{\alpha}} \frac{\bE[\tau^{2\gamma}]^{1/(\alpha\gamma)}}{\bE[\tau^2]^{1/\alpha}} \left\|f - \bE_\tau[f]\right\|_{L^{\alpha\beta}(M, \bR)}\\
 &\qquad + \frac{\bE[\tau^{2\gamma}]^{1/\gamma}}{\bE[\tau^2]} \left\|f - \bE_\tau[f]\right\|_{L^{\alpha\beta}(M, \bR)}^\alpha.
\end{aligned}
\end{equation}

Next, we apply Inequality \eqref{eqCrap2} to the left hand side of \eqref{eqMoment}:
\begin{align*}
\left\|f^\alpha - \bE_\tau[f]^\alpha\right\|_{L^\beta(M, \bR)}
 &\leq \left\|\alpha \bE_\tau[f]^{\alpha-1} \left|f - \bE_\tau[f]\right| + \left|f - \bE_\tau[f]\right|^\alpha\right\|_{L^\beta(M, \bR)}\\
 &\leq \alpha \bE_\tau[f]^{\alpha-1}\left\|f - \bE_\tau[f]\right\|_{L^\beta(M, \bR)} + \left\|\left|f - \bE_\tau[f]\right|^\alpha\right\|_{L^\beta(M, \bR)}\\
 &\leq \alpha \bE_\tau[f^\alpha]^{\frac{\alpha-1}{\alpha}}\left\|f - \bE_\tau[f]\right\|_{L^\beta(M, \bR)} + \left\|f - \bE_\tau[f]\right\|_{L^{\alpha\beta}(M, \bR)}^\alpha.
\end{align*}
Finally, combinig with \eqref{eqMean}, we get
\begin{align*}
\left\|f^\alpha - \bE_\tau[f^\alpha]\right\|_{L^\beta(M, \bR)}
 &\leq \left\|f^\alpha - \bE_\tau[f]^\alpha\right\|_{L^\beta(M, \bR)} + \left|\bE_\tau[f^\alpha] - \bE_\tau[f]^\alpha\right|\\
 &\leq \alpha \bE_\tau[f]^{\alpha-1}\left\|f - \bE_\tau[f]\right\|_{L^\beta(M, \bR)} + \left\|f - \bE_\tau[f]\right\|_{L^{\alpha\beta}(M, \bR)}^\alpha\\
 & \qquad+ \alpha \bE_\tau[f^\alpha]^{\frac{\alpha-1}{\alpha}} \frac{\bE[\tau^{2\gamma}]^{1/(\alpha\gamma)}}{\bE[\tau^2]^{1/\alpha}} \left\|f - \bE_\tau[f]\right\|_{L^{\alpha\beta}(M, \bR)}\\
 & \qquad+ \frac{\bE[\tau^{2\gamma}]^{1/\gamma}}{\bE[\tau^2]} \left\|f - \bE_\tau[f]\right\|^\alpha_{L^{\alpha\beta}(M, \bR)}\\
 &\leq \alpha \bE_\tau[f^\alpha]^{\frac{\alpha-1}{\alpha}} \left(1+ \frac{\bE[\tau^{2\gamma}]^{1/(\alpha\gamma)}}{\bE[\tau^2]^{1/\alpha}}\right) \left\|f - \bE_\tau[f]\right\|_{L^{\alpha\beta}(M, \bR)}\\
 &\qquad + \left(1+\frac{\bE[\tau^{2\gamma}]^{1/\gamma}}{\bE[\tau^2]}\right) \left\|f - \bE_\tau[f]\right\|^\alpha_{L^{\alpha\beta}(M, \bR)}.
\end{align*}
This ends the proof of the lemma.
\end{proof}
In view of Estimate \eqref{eqEstimateLich2}, we choose $f = \phi^{\frac{N}{2}+1}$,
$\alpha = \frac{N}{N/2+1} = \frac{n}{n-1}$ and $\beta = \frac{N}{2}+1$
so $\alpha\beta = N$ and $\gamma = 2\frac{n-1}{n}$. We remind the reader that,
according to our notation, $\bE_\tau[\phi^N] = c'$ and $\psi' = \phi^N-c'$.
We obtain:
\begin{equation}\label{eqMomentPhi}
\begin{aligned}
\left\|\psi'\right\|_{L^\beta(M, \bR)}
 &\leq \alpha (c')^{\frac{\alpha-1}{\alpha}} \left[1+ \frac{\|\tau\|_{L^{2\gamma}(M, \bR)}^{2/\alpha}}{\|\tau\|_{L^2(M, \bR)}^{2/\alpha}}\right] \left\|\phi^{\frac{N}{2}+1} - \bE_\tau[\phi^{\frac{N}{2}+1}]\right\|_{L^N(M, \bR)}\\
 & \qquad + \left[1+\frac{\|\tau\|_{L^{2\gamma}(M, \bR)}^2}{\|\tau\|_{L^2(M, \bR)}^2}\right] \left\|\phi^{\frac{N}{2}+1} - \bE_\tau[\phi^{\frac{N}{2}+1}]\right\|^\alpha_{L^N(M, \bR)}.
\end{aligned}
\end{equation}
Note that $2\gamma = 4\frac{n-1}{n} \leq n$, since, multiplying by $n$,
this inequality is nothing but $(n-2)^2 \geq 0$. As we assumed $\tau \in W^{1, t}(M, \bR) \subset L^n(M, \bR)$,
all norms of $\tau$ appearing in Estimate \eqref{eqMomentPhi} are finite.

Returning to Estimate \eqref{eqEstimateLich2}, remark that
\[
 \frac{n-1}{n} \int_M \tau^2 \phi^N d\mu^g = \frac{n-1}{n} \int_M \tau^2 ((c')^2 + 2 c' \psi' + (\psi')^2) d\mu^g = \frac{n-1}{n} \int_M \tau^2 ((c')^2 + (\psi')^2) d\mu^g
\]
due to our choice of decomposition. As a consequence, Estimate \eqref{eqEstimateLich2}
implies
\begin{equation}\label{eqEstimateLich3}
\left\lbrace
\begin{aligned}
 \frac{3n-2}{n-1} \frac{1}{s} \left\|\phi^{\frac{N}{2}+1} - \bE_\tau\left[\phi^{\frac{N}{2}+1}\right]\right\|_{L^N(M, \bR)}^2 & \leq \int_M \left|\sigma + \frac{\bL W}{2 \eta}\right|^2 d\mu^g,\\
 \frac{n-1}{n} \int_M \tau^2 d\mu^g  (c')^2 & \leq \int_M \left|\sigma + \frac{\bL W}{2 \eta}\right|^2 d\mu^g.
\end{aligned}
\right.
\end{equation}
Hence,
the first line of Estimate \eqref{eqEstimateLich3} together with \eqref{eqMomentPhi} imply
\begin{equation}\label{eqMomentPhi2}
\left\|\psi'\right\|_{L^{\frac{N}{2} + 1}}
 \leq c_1 (c')^{1/n} \left(\int_M \left|\sigma + \frac{\bL W}{2 \eta}\right|^2 d\mu^g\right)^{1/2} + c_2 \left(\int_M \left|\sigma + \frac{\bL W}{2 \eta}\right|^2 d\mu^g\right)^{\frac{n}{2(n-1)}}.
\end{equation}
for some constants $c_1, c_2$ depending only on $(M, g, \tau)$ and $p$.
The right hand side of Estimates \eqref{eqEstimateLich3} can be bounded from above as follows:
\begin{equation}\label{eqEstimateZ}
\begin{aligned}
\int_M \left|\sigma + \frac{\bL W}{2 \eta}\right|^2 d\mu^g
 &= \int_M \left|\sigma + c \frac{\bL \Wbar}{2 \eta} + \frac{\bL W_\psi}{2 \eta}\right|^2 d\mu^g\\
 &= \int_M \left[|\sigma|^2 + 2c \left\<\sigma, \frac{\bL \Wbar}{2 \eta}\right\> + c^2 \left|\frac{\bL \Wbar}{2 \eta}\right|^2\right.\\
 &\qquad\qquad\left. + \left\<\frac{\bL W_\psi}{2 \eta}, 2\sigma + 2c \frac{\bL \Wbar}{2\eta} + \frac{\bL W_\psi}{2 \eta}\right\> \right]d\mu^g\\
 &\leq x^2 + 2 c x A_1 + c^2 A_1^2\\
 &\qquad\qquad + \left\|\frac{\bL W_\psi}{2 \eta}\right\|_{L^2(M, \bR)} \left(2 x + 2 c A_1 + \left\|\frac{\bL W_\psi}{2 \eta}\right\|_{L^2(M, \bR)}\right).
\end{aligned}
\end{equation}
with
\[
 x \definedas \left(\int_M |\sigma|^2 d\mu^g\right)^{1/2},\quad A_1 \definedas \left\|\frac{\bL \Wbar}{2 \eta}\right\|_{L^2(M, \bR)}.
\]

From what we saw above, controlling the $L^2$-norm of the right hand side
in Estimate \eqref{eqEstimateLich3} (which is the best thing we can do if
we insists on imposing restrictions on the $L^2$-norm of $\sigma$ only)
gives no more than an $L^{\frac{N}{2}+1}$-control on $\psi'$. This is why
the restriction for the set $C_0$ only concerned this norm. Moreover,
since $W_\psi$ solves
\[
 \DeltaL W_\psi = \frac{n-1}{n} \psi \nabla\tau,
\]
the best we can say from Proposition \ref{propVector} is that $W_\psi$
is controlled in the $W^{2, q}$-norm for $q = \frac{2n(n-1)}{n^2+2n-4} + O(p-n)$.
From the Sobolev embedding theorem, it follows that
\[
 \frac{\bL W_\psi}{2\eta} \in L^{q'}(M, S_2M)
\]
with $q' = 2 \frac{n(n-1)}{n^2-2} + O(p-n)$. If $p$ is too close to $n$,
we have $q' < 2$. As a consequence, we need to reinforce our assumption
on $\nabla\tau$. To control the $L^2$-norm of $\frac{\bL W\psi}{2\eta}$
we need to impose that $\tau \in W^{1,t_0}(M, \bR)$ where $t_0$ is defined
in \eqref{eqDefT0}.
Indeed, from H\"older's inequality, we then have that
\[
 \left\|\psi \nabla\tau\right\|_{L^v(M, TM)} \leq \|\psi\|_{L^{\frac{N}{2}+1}(M, \bR)} \|\nabla\tau\|_{L^{t_0}(M, TM)}
\]
with $v = 2n/(n+2)$. So, from Proposition \ref{propVector}, there
exists a constant $\Lambda > 0$ so that
\[
 \left\|W_\psi\right\|_{W^{2, v}(M, TM)} \leq \Lambda \|\psi\|_{L^{\frac{N}{2}+1}(M, \bR)} \|\nabla\tau\|_{L^{t_0}(M, TM)}
\]
and, using the Sobolev embedding theorem, we get
\[
 \left\|\frac{\bL W_\psi}{2\eta}\right\|_{L^2(M, \bR)} \leq \Lambda' \|\psi\|_{L^{\frac{N}{2}+1}(M, \bR)} \|\nabla\tau\|_{L^{t_0}(M, TM)}
\]
for some constant $\Lambda' = \Lambda'(M, g, \tau, \eta)$.
For reasons that will become apparant later, we need to impose
$\tau \in W^{1, t}(M, \bR)$ with $t > t_0$. We can now return to
Estimate \eqref{eqEstimateZ}. From what we just saw, we have
\begin{equation}\label{eqEstimateZ2}
\begin{aligned}
\int_M \left|\sigma + \frac{\bL W}{2 \eta}\right|^2 d\mu^g
 &\leq x^2 + 2 c x A_1 + c^2 A_1^2\\
 &\qquad\qquad + \Lambda' \|\psi\|_{L^{\frac{N}{2}+1}} \|\nabla\tau\|_{L^{t_0}} \left(2 x + 2 c A_1 + \Lambda' \|\psi\|_{L^{\frac{N}{2}+1}} \|\nabla\tau\|_{L^{t_0}}\right)\\
 &\leq x^2 + 2 c_{\max} x A_1 + c_{\max}^2 A_1^2 + \Lambda'' r \left(2 x + 2 c A_1 + \Lambda'' r\right),
\end{aligned}
\end{equation}
where we set
\[
 \Lambda'' \definedas \Lambda' \|\nabla\tau\|_{L^{t_0}}.
\]
Defining
\[
 A_0^2 \definedas \frac{n-1}{n} \int_M \tau^2 d\mu^g,
\]
Estimates \eqref{eqEstimateLich3} and \eqref{eqMomentPhi2} imply
\begin{equation}\label{eqEstimateLich4}
\left\lbrace
\begin{aligned}
 A_0^2 (c')^2 &\leq f(x, c_{\max}, r)\\
 \left\|\psi'\right\|_{L^{\frac{N}{2} + 1}(M, \bR)}
 &\leq c_1 c_{\max}^{1/n} f(x, c_{\max}, r)^{1/2} + c_2 f(x, c_{\max}, r)^{\frac{n}{2(n-1)}},
\end{aligned}
\right.
\end{equation}
where
\begin{equation}\label{eqDefF}
 f(x, c_{\max}, r) \definedas x^2 + 2 c_{\max} x A_1 + c_{\max}^2 A_1^2 + \Lambda'' r \left(2 x + 2 c A_1 + \Lambda'' r\right)
\end{equation}
The set $C_0$ introduced in \eqref{eqDefC} will be stable provided that
we choose $c_{\max}$ and $r$ such that
\begin{equation}\label{eqStable}
\left\lbrace
\begin{aligned}
 A_0^2 c_{\max}^2 &\geq f(x, c_{\max}, r)\\
 r &\geq c_1 c_{\max}^{1/n} f(x, c_{\max}, r)^{1/2} + c_2 f(x, c_{\max}, r)^{\frac{n}{2(n-1)}}
\end{aligned}
\right.
\end{equation}
since these conditions immediately imply that $c' \leq c_{\max}$ and $\|\psi'\|_{L^{\frac{N}{2}+1}(M, \bR)} \leq r$.
To find a pair $(c_{\max}, r)$ satisfying \eqref{eqStable}, we set
\[
 c_{\max} = a x,\quad r = b x^{\frac{n}{n-1}}
\]
for some positive constants $a, b$ to be chosen later. This allows to keep
track of the order of magnitude of both components of $\phi^N$ as $x$ tends to zero.
The system \eqref{eqStable} can be rewritten
\begin{equation}\label{eqStable2}
\left\lbrace
\begin{aligned}
 A_0^2 a^2 &\geq (1+a A_1)^2 + O(x^{\frac{1}{n-1}})\\
 b &\geq c_2 (1+a A_1)^{\frac{n}{n-1}} + O(x^{\frac{1}{n-1}}),
\end{aligned}
\right.
\end{equation}
where the big O terms depend on $a$ and $b$.
The idea is now to replace inequalities by equalities and use the
implicit function theorem. Namely, when $x = 0$, the system
\begin{equation}\label{eqStable3}
\left\lbrace
\begin{aligned}
 A_0^2 a^2 &= (1+a A_1)^2\\
 b &= c_2 (1+a A_1)^{\frac{n}{n-1}},
\end{aligned}
\right.
\end{equation}
admits a solution, namely $a_0 = 1/(A_0-A_1)$ and $b_0 = c_2 (1+a_0 A_1)^{\frac{n}{n-1}}$
and the linearization of the system \eqref{eqStable3} has no non-trivial solution.
As a consequence, for small $x$ the system \eqref{eqStable2}
admits a solution in a vicinity of $(a_0, b_0)$.

We pause at this point and summarize what we have proven so far:
\begin{proposition}\label{propSmallTT}
 Assume that $g \in W^{2, p/2}(M, S_2M)$, $\sigma \in L^2(M, \bR)$ and $\tau \in W^{1,t_0}(M, \bR)$,
where $t_0 = \frac{2n(n-1)}{3n-2}$. Then provided that
\[
 x \definedas \left(\int_M |\sigma|^2 d\mu^g\right)^{1/2}
\]
 is small enough, there exist constants $c_{\max} > 0$ and $r > 0$ such
that the set $C_0$ defined in \eqref{eqDefC} is stable for the mapping $\Psi$.
\end{proposition}

We are not yet in a position to apply Schauder's fixed point theorem since
$C_0$ is not bounded. So, in what follows, we use a bootstrap argument to
find nested closed subsets $C_k$ (i.e. such that $C_{k+1} \subset C_k$)
so that $\Psi(C_k) \subset C_{k+1}$ eventually getting a bounded closed
set. This point is inspired by \cite[Proposition 4.6]{GicquaudNguyen}.\\

We construct sequences $(q_i)$, $(k_i)$, $(r_i)$, $(R_i)$ as follows.
We choose $q_0 = \frac{N}{2}+1$. There exists a constant $R_0 > 0$
such that
\[
 C_0 \subset \{u \in L^{q_0}(M, \bR), \|u\|_{L^{q_0}(M, \bR)} \leq R_0\}.
\]
Assume now that for some $i \geq 0$, $q_i$ and $R_i$ are knowns
(we just defined $q_0$ and $R_0$). Then, from Young's inequality,
we have that for all $u \in C_i$,
$\left\|u \nabla\tau\right\|_{L^{c_i}(M, \bR)} \lesssim R_i$ where $c_i$ satisfies
$\frac{1}{c_i} = \frac{1}{q_i} + \frac{1}{t}$. Here, the notation
$A \lesssim B$ means that there exists a constant $C > 0$
that may vary from line to line but independent of $u$ such that
$A \leq CB$.

By Proposition \ref{propVector}, we have
\[
 \left\|W\right\|_{W^{2, c_i}} \lesssim R_i
\]
for all $W$ solving \eqref{eqDefPsi1}. From the
Sobolev embedding theorem, we get that
\[
 \left\|\frac{\bL W}{2\eta}\right\|_{L^{r_i}(M, S_2M)} \lesssim R_i
\]
where $r_i$ is given by
\[
 \frac{1}{r_i} = \frac{1}{c_i} - \frac{1}{n} = \frac{1}{q_i} + \frac{1}{t} - \frac{1}{n}.
\]
We now multiply the Lichnerowicz equation by $\phi^{N+1+2 k_i}$
for some $k_i$ to be chosen later and integrate over $M$. We get
\[
\frac{4(n-1)}{n-2} \int_M \left\< d \phi^{N+1+2 k_i}, d\phi\right\> d\mu^g
 \leq \int_M \left|\sigma + \frac{\bL W}{2\eta}\right|^2 \phi^{2 k_i} d\mu^g,
\]
or, equivalently,
\begin{equation}\label{eqBoundGradient}
\frac{4(n-1)}{n-2} \frac{N+1+2 k_i}{\left(\frac{N}{2}+1+k_i\right)^2} \int_M \left| d \phi^{\frac{N}{2}+1+ k_i}\right|^2 d\mu^g
 \leq \int_M \left|\sigma + \frac{\bL W}{2\eta}\right|^2 \phi^{2 k_i} d\mu^g.
\end{equation}
Using H\"older's inequality, we have that
\begin{align*}
\int_M \left|\sigma + \frac{\bL W}{2\eta}\right|^2 \phi^{2 k_i} d\mu^g
 &\leq \left\|\sigma + \frac{\bL W}{2\eta}\right\|^2_{L^{r_i}(M, S_2M)} \left\|\phi^{2 k_i}\right\|_{L^{\frac{r_i}{r_i-2}}(M, \bR)}\\
 &\leq \left\|\sigma + \frac{\bL W}{2\eta}\right\|^2_{L^{r_i}(M, S_2M)} \left\|\phi^N\right\|^{\frac{2 k_i}{N}}_{L^{\frac{2 k_i}{N}\frac{r_i}{r_i-2}}(M, \bR)}.
\end{align*}
We now choose $k_i$ such that
\[
 \frac{2 k_i}{N}\frac{r_i}{r_i-2} = q_i,
\]
namely,
\[
 k_i = (N-1) q_i - N \left(\frac{q_i}{t}+1\right).
\]
We apply the Sobolev embedding theorem: for some constant $s_i$,
we have that
\begin{align*}
\left(\int_M \left| \phi^{\frac{N}{2}+1+ k_i}\right|^N d\mu^g\right)^{2/N}
 &\leq s_i \int_M \left|\sigma + \frac{\bL W}{2\eta}\right|^2 \phi^{2 k_i} d\mu^g + \int_M \phi^{N+2+2k_i} d\mu^g,\\
\left\|\phi^N\right\|^{N+2+2k_i}_{L^{\frac{N}{2}+1+k_i}(M, \bR)}
 &\leq s_i \left\|\sigma + \frac{\bL W}{2\eta}\right\|^2_{L^{r_i}(M, S_2M)} \left\|\phi^N\right\|^{\frac{2 k_i}{N}}_{L^{q_i}(M, \bR)} + \left\|\phi^N\right\|_{L^{1 + \frac{2}{N} +\frac{2k_i}{N}}(M, \bR)}^{1 + \frac{2}{N}+ \frac{2k_i}{N}}.
\end{align*}
A straightforward calculation shows that $1 + \frac{2}{N} + \frac{2 k_i}{N} < q_i$,
so, since $g$ has volume one,
\[
\left\|\phi^N\right\|^{N+2+2k_i}_{L^{\frac{N}{2}+1+k_i}(M, \bR)}
 \leq s_i \left\|\sigma + \frac{\bL W}{2\eta}\right\|^2_{L^{r_i}(M, S_2M)} \left\|\phi^N\right\|^{\frac{2 k_i}{N}}_{L^{q_i}(M, \bR)} + \left\|\phi^N\right\|_{L^{q_i}(M, \bR)}^{1 + \frac{2}{N}+ \frac{2k_i}{N}}.
\]
Setting
\begin{equation}\label{eqRecurrence}
 q_{i+1} = \frac{N}{2}+1+ k_i = \left(N-1 - \frac{N}{t}\right) q_i - \frac{2}{n-2},
\end{equation}
we obtain that
\[
\left\|\phi^N\right\|^{N+2+2k_i}_{L^{q_{i+1}}(M, \bR)}
 \leq s_i \left\|\sigma + \frac{\bL W}{2\eta}\right\|^2_{L^{r_i}(M, S_2M)} \left\|\phi^N\right\|^{\frac{2 k_i}{N}}_{L^{q_i}(M, \bR)} + \left\|\phi^N\right\|_{L^{q_i}(M, \bR)}^{1 + \frac{2}{N}+ \frac{2k_i}{N}}.
\]
Since we know that $\phi^N \in C_i$, we have $\|\phi^N\|_{L^{q_i}(M, \bR)} \leq R_i$
it is then immediate that
\[
 \|\phi^N\|_{L^{q_{i+1}}} \leq R_{i+1}
\]
for some well chosen $R_{i+1}$ as we have bounded all the terms of the right
hand side.

Setting $C_{i+1} \definedas \{u \in C_i, \|u\|_{L^{q_{i+1}}(M, \bR)} \leq R_{i+1}\}$,
we have that $C_{i+1} \subset C_i$ and $\Psi(C_i) \subset C_{i+1}$.

We now study in more details the sequence $(q_i)$. It is defined by the
recurrence relation \eqref{eqRecurrence}. Let $\qbar$ denote the solution
to
\[
 \qbar = \left(N-1 - \frac{N}{t}\right) \qbar - \frac{2}{n-2},
\]
namely
\[
 \qbar = \frac{2}{n-2} \frac{t}{(N-2)t-N}.
\]
We have
\[
 q_i = \left(N-1 - \frac{N}{t}\right)^i (q_0 - \qbar) + \qbar.
\]
If $t \geq t_0$, we have
\[
 N-1 - \frac{N}{t} \geq \frac{n}{n-1} > 1.
\]
Yet if $t = t_0$, where $t_0$ is defined in \eqref{eqDefT0}, we have
$\qbar = q_0$ so the sequence $(q_i)$ is constant. This is where we
have to assume that $t > t_0$ to ensure $q_i \to \infty$. There is
an $i_0$ such that $q_{i_0+1} > p_0$. Then $C_{i_0+1}$ is bounded and closed
in $L^{p_0}(M, \bR)$.

Even more is true. Assume that $i_0$ has been chosen so that $q_{i_0+1} > \max\{p_0, N\}$.
Set $C \definedas \Psi(C_i)$. We claim
that $C$ is precompact in $L^{p_0}(M, \bR)$. Indeed, performing the analysis
following \eqref{eqBoundGradient} but without using the Sobolev embedding
theorem, we have that, for any $u = \phi^N \in \Psi(C_i)$,
\[
\left\|\phi^{q_{i+1}}\right\|_{W^{1, 2}(M, \bR)} \leq R'
\]
for some constant $R' > 0$. Let us denote $q = q_{i+1}$ for simplicity.

Let $(u_k)_k$, $u_k = \phi_k^N \in \Psi(C_i)$,
be any given sequence. Since $p_0 < q$, we have that
$\lambda \definedas N \frac{p_0}{q} < N$ so the embedding
$W^{1, 2}(M, \bR) \hookrightarrow L^\lambda(M, \bR)$ is compact.
As a consequence, there exists a subsequence $(u_{\omega(k)})_k$ of
$(u_k)_k$ such that $\phi_{\omega(k)}^q \to \phi_\infty^q$ in $L^\lambda(M, \bR)$.
We have $u_\infty = \phi_\infty^N \in L^{p_0}(M, \bR)$ so all we need to
do is to check that
\[
 u_{\omega(k)} \to u_\infty \text{ in } L^{p_0}(M, \bR).
\]
The idea is similar to the one for \eqref{eqMoment}, yet simpler.
Let $h: \bR_+ \to \bR_+$ denote the function
\[
 h(x) \definedas \frac{\left|x^N-1\right|^{p_0}}{\left|x^q-1\right|^\lambda}.
\]
Since we chose $q > N$, we have $\lambda = \frac{N}{q} p_0 < p_0$ and
\[
 h(x) \sim \frac{N^{p_0}}{q^\lambda} |x-1|^{p_0-\lambda}
\]
near $x = 1$. Since $h(x)$ tends to $1$ when $x$ goes to $0$ or to $\infty$,
we conclude that $h$ is bounded on $\bR_+$. Let $A > 0$ be an upper bound for $h$: $h(x) \leq A$
for all $x \in \bR_+$. We have that
\[
 \left|x^N-1\right|^{p_0} \leq A \left|x^q-1\right|^\lambda.
\]
Setting $x = \phi_i/\phi_\infty$ and multiplying by $\phi_\infty^{N p_0} = \phi_\infty^{\lambda q}$,
we have that
\[
 \left|\phi_i^N-\phi_\infty^N\right|^{p_0} \leq A \left|\phi_i^q-\phi_\infty^q\right|^\lambda.
\]
Integrating over $M$, we obtain
\[
 \left\|u_i - u_\infty\right\|_{L^{p_0}(M, \bR)} \leq A^{1/p_0} \left\|\phi_i^q-\phi_\infty^q\right\|_{L^\lambda(M, \bR)}^{\lambda/p_0}
\]
which shows that $u_i \to u_\infty$ in $L^{p_0}(M, \bR)$. We have proven that $C$
is (sequentially) precompact in $L^{p_0}(M, \bR)$.

Set $\Cbar \definedas \overline{\mathrm{conv}(C)}$ be the closed convex hull
of $C$. Then $\Cbar$ is compact, convex and $\Psi(\Cbar) \subset \Psi(C_{i_0}) \subset \Cbar$.
We can now apply the Schauder fixed point theorem to $\Psi$ and $\Cbar$ and
get the following theorem:

\begin{theorem}\label{thmSmallTT}
 Assume that $g \in W^{2, p/2}(M, S_2M)$, $\sigma \in L^p(M, S_2M)$, $\eta \in W^{1, p}(M, \bR)$, $\tau \in W^{1,t}(M, \bR)$
for some $t > t_0$, where $t_0$ is defined in \eqref{eqDefT0}. Then, provided
that
\[
 \int_M |\sigma|^2 d\mu^g
\]
is small enough (as given in Proposition \ref{propSmallTT}), there exists at
least one solution to the system \eqref{system}.
\end{theorem}

%%%%%%%%%%%%%%%%%%%%%%%%%%%%%%%%%%%%%%%%%%%%%%%%%%%%%%%%%%%%%%%%%%%%%%%%%
\providecommand{\bysame}{\leavevmode\hbox to3em{\hrulefill}\thinspace}
\providecommand{\MR}{\relax\ifhmode\unskip\space\fi MR }
\providecommand{\MRhref}[2]{%
  \href{http://www.ams.org/mathscinet-getitem?mr=#1}{#2}
}
\providecommand{\href}[2]{#2}

%%%%%%%%%%%%%%%%%%%%%%%%%%%%%%%%%%%%%%%%%%%%%%%%%%%%%%%%%%%%%%%%%%%%%%%%%


\begin{thebibliography}{10}

\bibitem{ArnowittDeserMisner}
Richard~L. Arnowitt, Stanley Deser, and Charles~W. Misner, \emph{{The Dynamics
  of general relativity}}, Gen. Rel. Grav. \textbf{40} (2008), 1997--2027.

\bibitem{BartnikIsenberg}
R.~Bartnik and J.~Isenberg, \emph{The constraint equations}, The {E}instein
  equations and the large scale behavior of gravitational fields, Birkh\"auser,
  Basel, 2004, pp.~1--38.

\bibitem{DahlGicquaudHumbert}
M.~Dahl, R.~Gicquaud, and E.~Humbert, \emph{A limit equation associated to the
  solvability of the vacuum {E}instein constraint equations by using the
  conformal method}, Duke Math. J. \textbf{161} (2012), no.~14, 2669--2697.

\bibitem{DiltsHolstKozarevaMaxwell}
J.~Dilts, M.~Holst, T.~Kozareva, and D.~Maxwell, \emph{Numerical bifurcation
  analysis of the conformal method}, arXiv:1710.03201.

\bibitem{GicquaudNgo}
R.~Gicquaud and Q.A. Ng\^{o}, \emph{A new point of view on the solutions to the
  {E}instein constraint equations with arbitrary mean curvature and small
  {TT}-tensor}, Class. Quantum Grav. \textbf{31} (2014), no.~19, 195014 (20pp).

\bibitem{GicquaudNguyen}
R.~Gicquaud and C.~Nguyen, \emph{Solutions to the {E}instein-scalar field
  constraint equations with a small {TT}-tensor}, Calc. Var. Partial
  Differential Equations \textbf{55} (2016), no.~2, Art. 29, 23. \MR{3466902}

\bibitem{HolstMaxwellMazzeo}
M.~Holst, D.~Maxwell, and R.~Mazzeo, \emph{Conformal fields and the structure
  of the space of solutions of the {E}instein constraint equations},
  arXiv:1711.01042.

\bibitem{HolstMeierNonUniqueness}
M.~Holst and C.~Meier, \emph{Non-uniqueness of solutions to the conformal
  formulation}, arXiv:1210.2156.

\bibitem{HNT1}
M.~Holst, G.~Nagy, and G.~Tsogtgerel, \emph{Far-from-constant mean curvature
  solutions of {E}instein's constraint equations with positive {Y}amabe
  metrics}, Phys. Rev. Lett. \textbf{100} (2008), no.~16, 161101, 4.

\bibitem{HNT2}
\bysame, \emph{Rough solutions of the {E}instein constraints on closed
  manifolds without near-{CMC} conditions}, Comm. Math. Phys. \textbf{288}
  (2009), no.~2, 547--613.

\bibitem{Isenberg}
J.~Isenberg, \emph{Constant mean curvature solutions of the {E}instein
  constraint equations on closed manifolds}, Class. Quantum Grav. \textbf{12}
  (1995), no.~9, 2249--2274.

\bibitem{LeeParker}
J.M. Lee and T.H. Parker, \emph{The {Y}amabe problem}, Bull. Amer. Math. Soc.
  (N.S.) \textbf{17} (1987), no.~1, 37--91.

\bibitem{MaxwellInitialData}
D.~Maxwell, \emph{Initial data in general relativity described by expansion,
  conformal deformation and drift}, arXiv:1407.1467.

\bibitem{MaxwellNonCMC}
\bysame, \emph{A class of solutions of the vacuum {E}instein constraint
  equations with freely specified mean curvature}, Math. Res. Lett. \textbf{16}
  (2009), no.~4, 627--645.

\bibitem{MaxwellConformalParameterization}
\bysame, \emph{A model problem for conformal parameterizations of the
  {E}instein constraint equations}, Comm. Math. Phys. \textbf{302} (2011),
  no.~3, 697--736. \MR{2774166 (2012b:53164)}

\bibitem{MaxwellConformalMethod}
\bysame, \emph{The conformal method and the conformal thin-sandwich method are
  the same}, Classical Quantum Gravity \textbf{31} (2014), no.~14, 145006, 34.
  \MR{3233274}

\bibitem{Nguyen2}
T.~C. Nguyen, \emph{Nonexistence and nonuniqueness results for solutions to the
  vacuum {E}instein conformal constraint equations}, submitted,
  arXiv:1507.01081.

\bibitem{Nguyen}
\bysame, \emph{Applications of fixed point theorems to the vacuum {E}instein
  constraint equations with non-constant mean curvature}, Ann. Henri Poincar\'e
  \textbf{17} (2016), no.~8, 2237--2263.

\bibitem{RingstromCauchyProblem}
H.~Ringstr{\"o}m, \emph{The {C}auchy problem in general relativity}, ESI
  Lectures in Mathematics and Physics, European Mathematical Society (EMS),
  Z\"urich, 2009. \MR{2527641 (2010j:83001)}

\bibitem{TrudingerMeasurable}
N.~S. Trudinger, \emph{Linear elliptic operators with measurable coefficients},
  Ann. Scuola Norm. Sup. Pisa (3) \textbf{27} (1973), 265--308. \MR{0369884}

\bibitem{Yap}
C.~K. Yap, \emph{Fundamental problems of algorithmic algebra}, Oxford
  University Press, New York, 2000. \MR{1740761}

\bibitem{YorkDecomposition}
James~W. York, Jr., \emph{Conformally invariant orthogonal decomposition of
  symmetric tensors on {R}iemannian manifolds and the initial-value problem of
  general relativity}, J. Mathematical Phys. \textbf{14} (1973), 456--464.
  \MR{0329562}

\end{thebibliography}
\end{document}